\documentclass[a4paper,11pt]{article}
\usepackage[utf8]{inputenc}

\usepackage{amsmath, amssymb, amsthm}
\usepackage[english]{babel}
\usepackage[top=2.5cm,bottom=2.5cm,left=2.5cm,right=2.5cm]{geometry}
\usepackage{hyperref}
\usepackage{tikz}
\usepackage{url}
\usepackage{xcolor}

\usetikzlibrary{calc,intersections}

\title{Higgledy-piggledy sets in projective spaces of small dimension}
\author{Lins Denaux \\ {\it Ghent University}}
\date{}

\newtheorem{thm}{Theorem}[section]
\newtheorem*{mainres}{Main Results}
\newtheorem{prop}[thm]{Proposition}
\newtheorem{lm}[thm]{Lemma}
\newtheorem{crl}[thm]{Corollary}

\theoremstyle{definition}
\newtheorem{df}[thm]{Definition}
\newtheorem{nt}[thm]{Notation}
\newtheorem{rmk}[thm]{Remark}
\newtheorem{conf}[thm]{Configuration}
\newtheorem{prob}[thm]{Open Problem}
\newtheorem{snip}[thm]{Code Snippet}

\newcommand{\n}{N} 
\newcommand{\len}{n} 
\newcommand{\red}{r} 

\newcommand{\NN}{\mathbb{N}}
\newcommand{\FF}{\mathbb{F}}

\newcommand{\vspan}[1]{\left\langle#1\right\rangle}
\newcommand{\set}[1]{\left\{#1\right\}}
\newcommand{\sett}[2]{\left\{#1\,:\,#2\right\}}
\newcommand{\V}[2]{V\!\mspace{-2mu}\left(#1,#2\right)}
\newcommand{\pg}[2]{\textnormal{PG}\!\left(#1,#2\right)}
\newcommand{\pgTitle}[2]{\textnormal{\textbf{PG}}\!\left(#1,#2\right)}

\newcommand{\Strong}{\mathcal{B}}
\newcommand{\Sat}{\mathcal{S}}
\newcommand{\satbound}{s_q}
\newcommand{\satboundpow}[1]{s_{q^{#1}}}
\newcommand{\lenfunc}{\ell_q}
\newcommand{\lenfuncpow}[1]{\ell_{q^{#1}}}

\renewcommand{\geq}{\geqslant}
\renewcommand{\leq}{\leqslant}
\renewcommand{\rho}{\varrho}

\begin{document}

\maketitle

\begin{abstract}
    This work focuses on higgledy-piggledy sets of $k$-subspaces in $\pg{\n}{q}$, i.e.\ sets of projective subspaces that are `well-spread-out'.
    More precisely, the set of intersection points of these $k$-subspaces with any $(\n-k)$-subspace $\kappa$ of $\pg{\n}{q}$ spans $\kappa$ itself.
    
    We highlight three methods to construct small higgledy-piggledy sets of $k$-subspaces and discuss, for $k\in\set{1,\n-2}$, `optimal' sets that cover the smallest possible number of points.
    
    Furthermore, we investigate small non-trivial higgledy-piggledy sets in $\pg{\n}{q}$, $\n\leq5$.
    Our main result is the existence of six lines of $\pg{4}{q}$ in higgledy-piggledy arrangement, two of which intersect.
    Exploiting the construction methods mentioned above, we also show the existence of six planes of $\pg{4}{q}$ in higgledy-piggledy arrangement, two of which maximally intersect, as well as the existence of two higgledy-piggledy sets in $\pg{5}{q}$ consisting of eight planes and seven solids, respectively.
    
    Finally, we translate these geometrical results to a coding- and graph-theoretical context.
\end{abstract}

{\it Keywords:} Covering codes, Cutting blocking sets, Higgledy-piggledy sets, Minimal codes, Projective spaces, Resolving sets, Saturating sets, Strong blocking sets.

{\it Mathematics Subject Classification:} $05$B$25$, $94$B$05$, $51$E$20$, $51$E$21$.

\section{Introduction}

The main topic of this article concerns small sets of projective subspaces that are `well-spread-out' (copying a well-put description of \cite{FancsaliSziklai2,FancsaliSziklai3}).
The existence of infinite families of such combinatorial objects implies the existence of e.g.\ minimal codes and covering codes of relatively small length (see Section \ref{Sect_Fruits}).
Before going into detail concerning the significance of these structures and their applications, we first describe the setting in which we will work and introduce the necessary preliminaries.

Throughout this work, we generally assume that $\n\in\NN$ and that $q$ is a prime power.
The Galois field of order $q$ will be denoted by $\FF_q$ and the Desarguesian projective space of (projective) dimension $\n$ over $\FF_q$ will be denoted by $\pg{\n}{q}$.
We refer to \cite{Hirschfeld,HirschfeldThas} for an extensive overview on the topic of finite geometry.

\bigskip
One type of structure that is being thoroughly investigated in the literature are \emph{blocking sets}.
We adopt the definition used in \cite{DeBeuleStorme}.

\begin{df}\label{Def_BlockingSet}
    Let $k\in\set{0,1,\dots,\n-1}$.
    A $k$\emph{-blocking set} of $\pg{\n}{q}$ is a point set that meets every $(\n-k)$-dimensional subspace.
    A $1$-blocking set will simply be called a \emph{blocking set}.
\end{df}

A $k$-subspace is the easiest (and smallest) example of a $k$-blocking set of $\pg{\n}{q}$ \cite{BoseBurton}.
A natural generalisation of a $k$-blocking set of $\pg{\n}{q}$ is the concept of a $t$-fold $k$-blocking set, $t\in\NN$, which is a point set of $\pg{\n}{q}$ that meets every $(\n-k)$-dimensional subspace in at least $t$ points.
Obviously, any set of $t$ pairwise disjoint $k$-subspaces is a $t$-fold $k$-blocking set.

\begin{df}
    Let $k\in\set{0,1,\dots,\n-1}$.
    A \emph{strong $k$-blocking set} is a point set that meets every $(\n-k)$-dimensional subspace $\kappa$ in a set of points spanning $\kappa$.
    A strong $1$-blocking will simply be called a \emph{strong blocking set}.
\end{df}

The concept of a strong $k$-blocking set was introduced in \cite[Definition $3.1$]{DavydovEtAl}.
However, these are also known as \emph{generator sets} (\cite[Definition $2$]{FancsaliSziklai2}) and \emph{cutting blocking sets} (\cite[Definition $3.4$]{BoniniBorello}) in case $k=1$.

Note that any strong $k$-blocking set is necessarily an $(\n-k+1)$-fold $k$-blocking set, although the converse is generally false.
Following this line of thought, one could wonder if a strong $k$-blocking set could be constructed by considering all points lying in the union of a certain number of well-chosen $k$-subspaces. Although sporadic examples of such point sets were already presented in \cite{DavydovEtAl}, this idea was thoroughly investigated in \cite{FancsaliSziklai2,HegerPatkosTakats} for $k=1$ and later generalised in \cite{FancsaliSziklai3} to arbitrary $k$.

\begin{df}[Higgledy-piggledy set of $k$-subspaces]
    Let $k\in\set{0,1,\dots,\n-1}$ and suppose that $\mathcal{K}$ is a set of $k$-subspaces in $\pg{\n}{q}$.
    If the set of all points lying in at least one subspace of $\mathcal{K}$ is a strong $k$-blocking set, then the elements of $\mathcal{K}$ are said to be in \emph{higgledy-piggledy arrangement} and the set $\mathcal{K}$ itself is said to be a \emph{higgledy-piggledy set of $k$-subspaces}.
\end{df}

One often excludes the trivial cases $k\in\set{0,\n-1}$.
After all, any set of $\n+1$ points spanning the whole space $\pg{\n}{q}$ is a higgledy-piggledy point set of smallest size.
Conversely (or by duality, see Proposition \ref{Prop_Duality}), any set of $\n+1$ hyperplanes with the property that no point lies in all of them, is a higgledy-piggledy set of hyperplanes of smallest size.

\bigskip
If $1\leq k\leq\n-2$, however, it is generally not an easy task to find small higgledy-piggledy sets of $k$-subspaces.
The following `almost-equivalent' condition was first derived for sets of lines in \cite{FancsaliSziklai2} and later generalised to sets of $k$-subspaces in \cite{FancsaliSziklai3}, and will prove to be a great tool to ease the search for higgledy-piggledy sets.

\begin{thm}[{\cite[Theorem $4$ and Proposition $5$]{FancsaliSziklai3}}]\label{Res_CharHigPig}
    Let $k\in\set{0,1,\dots,\n-1}$ and suppose that $\mathcal{K}$ is a set of $k$-subspaces in $\pg{\n}{q}$.
    If no $(\n-k-1)$-subspace meets each element of $\mathcal{K}$, then $\mathcal{K}$ is a higgledy-piggledy set of $k$-subspaces.
    If $|\mathcal{K}|\leq q$, the converse holds as well.
\end{thm}

As one generally wishes to construct higgledy-piggledy sets of small size, \textbf{lower bounds} on the size of such sets were determined to reveal which sizes would (theoretically) be optimal.
A lower bound on higgledy-piggledy line sets was determined in \cite{FancsaliSziklai2} for $q$ large enough, and very recently strengthened in \cite{HegerNagy} to all values of $q$.

\begin{thm}[{\cite[Theorem $3.12$]{HegerNagy}}]
    A higgledy-piggledy line set of $\pg{\n}{q}$ contains at least $\n+\left\lfloor\frac{\n}{2}\right\rfloor-\left\lfloor\frac{\n-1}{q}\right\rfloor$ elements.
\end{thm}

Based on the reasoning behind \cite[Theorem $14$]{FancsaliSziklai2}, the authors of \cite{FancsaliSziklai3} inductively determined a lower bound on the size of general higgledy-piggledy sets of $k$-subspaces.

\begin{thm}[{\cite[Theorem $20$]{FancsaliSziklai3}}]\label{Res_LowerBound}
    Let $k\in\set{0,1,\dots,\n-1}$.
    A higgledy-piggledy set of $k$-subspaces in $\pg{\n}{q}$ contains at least $\min\set{q,\sum_{i=0}^k\left\lfloor\frac{\n-k+i}{i+1}\right\rfloor}+1$ elements.
\end{thm}

The latter theorem can in fact be improved for $k>\frac{\n-1}{2}$ if one takes the duality of the projective space into account.

\begin{thm}\label{Thm_LowerBound}
    Let $k\in\set{0,1,\dots,\n-1}$.
    A higgledy-piggledy set of $k$-subspaces in $\pg{\n}{q}$ contains at least
    \[
        \min\set{q,\max\set{(k+1)+\sum_{i=1}^{k+1}\left\lfloor\frac{\n-k-1}{i}\right\rfloor,(\n-k)+\sum_{i=1}^{\n-k}\left\lfloor\frac{k}{i}\right\rfloor}}+1
    \]
    elements.
\end{thm}
\begin{proof}
    This follows immediately from Theorem \ref{Res_LowerBound} and Proposition \ref{Prop_Duality}.
    We have rewritten the formula to emphasize its duality.
\end{proof}

The main topic of this article concerns the flip side of the coin, namely the search for tighter \textbf{upper bounds} on the size of the smallest possible higgledy-piggledy sets of $k$-subspaces in $\pg{\n}{q}$.
This is naturally done by constructing small higgledy-piggledy sets, ideally with a size as close as possible to the theoretical lower bound.

A naive but interesting example of a general higgledy-piggledy line set is the \emph{tetrahedron}, first mentioned in \cite[Theorem $6$]{DavydovOstergard}.
This is a set of lines obtained by simply pairwise connecting $\n+1$ points of $\pg{\n}{q}$ that span the whole space, resulting in a set of $\frac{\n(\n+1)}{2}$ lines in higgledy-piggledy arrangement.
Several years later, smaller higgledy-piggledy line sets were found and subsequently generalised.

\begin{thm}[{\cite[Theorem $24$]{FancsaliSziklai2}}]\label{Res_HigPigExistence}
    If $q\geq2\n-1$, then there exist $2\n-1$ pairwise disjoint lines of $\pg{\n}{q}$ in higgledy-piggledy arrangement.
\end{thm}

\begin{thm}{\cite[Proposition $10$ and Subsection $3.4$]{FancsaliSziklai3}}\label{Res_HigPigExistenceGeneral}
    Let $k\in\set{0,1,\dots,\n-1}$.
    If $q>\n+1$, then there exist $(\n-k)(k+1)+1$ distinct $k$-subspaces of $\pg{\n}{q}$ in higgledy-piggledy arrangement.
\end{thm}

Although the theorems above present very strong upper bounds on the size of the smallest possible higgledy-piggledy sets of $k$-subspaces, the case of $q$ small is neglected.
Using a probabilistic approach, the authors of \cite[Theorems $4.1$ and $6.5$]{HegerNagy} very recently obtained new upper bounds for all values of $q$; one can check that their upper bounds improve the results above if and only if $q\leq\n+1$.
Especially their thorough investigation of the binary case deserves to be mentioned.

\bigskip
Besides these general results, sporadic examples of higgledy-piggledy sets can be found in the literature.

\begin{thm}\ \label{Res_SporadicExamples}
    \begin{enumerate}
        \item \cite[Theorem $3.7$, Example $9$]{DavydovEtAl,FancsaliSziklai2} There exist four pairwise disjoint lines of $\pg{3}{q}$ in higgledy-piggledy arrangement.
        \item \cite[Proposition $12$]{BartoliKissMarcuginiPambianco} If $q>36086$ is no power of $2$ or $3$, then there exist six pairwise disjoint lines of $\pg{4}{q}$ in higgledy-piggledy arrangement.
        \item \cite[Theorem $3.15$]{BartoliCossidenteMarinoPavese} There exist seven pairwise disjoint lines of $\pg{5}{q}$ in higgledy-piggledy arrangement.
    \end{enumerate}
\end{thm}

The authors of \cite{DavydovEtAl} also prove the existence of nine planes of $\pg{4}{q}$ in higgledy-piggledy arrangement.
However, if $q\geq7$, Theorem \ref{Res_HigPigExistenceGeneral} improves this result, as it implies the existence of seven planes of $\pg{4}{q}$ in higgledy-piggledy arrangement.

\bigskip
In this work, we focus on small non-trivial higgledy-piggledy sets of $\pg{4}{q}$ and $\pg{5}{q}$.
The article is organised as follows.
In Section \ref{Sect_ConstructionsAndExtremalSets}, we first discuss three general construction methods to obtain small higgledy-piggledy sets of $\pg{\n}{q}$.
Secondly, we bundle some curious observations concerning general higgledy-piggledy line sets and sets of $(\n-2)$-subspaces.

Section \ref{Sect_HigPig4} and \ref{Sect_HigPig5} are devoted to the search for small non-trivial higgledy-piggledy sets in $\pg{4}{q}$ and $\pg{5}{q}$, respectively.
The following summarises all main results that will be presented throughout this work.

\begin{mainres}
    For all prime powers $q$, the following higgledy-piggledy sets exist.
    \begin{enumerate}
        \item six lines of $\pg{4}{q}$, two of which intersect (Theorem \ref{Thm_SixHigPigLines4}),
        \item six planes of $\pg{4}{q}$, two of which intersect in a line (Corollary \ref{Crl_SixHigPigPlanes4}),
        \item eight pairwise disjoint planes of $\pg{5}{q}$ (Theorem \ref{Thm_EightHigPigPlanes5}), and
        \item seven solids of $\pg{5}{q}$, $q\geq7$, pairwise intersecting in a line (Corollary \ref{Crl_SevenHigPigSolids5}).
    \end{enumerate}
\end{mainres}

Finally, Section \ref{Sect_Fruits} describes some coding- and graph-theoretical results that arise due to the existence of the above higgledy-piggledy sets.

\section{Construction methods and optimal higgledy-piggledy sets}\label{Sect_ConstructionsAndExtremalSets}

The first part of this section is dedicated to the discussion of three techniques to construct higgledy-piggledy sets of $k$-subspaces: \emph{projection}, \emph{dualisation} and \emph{field reduction}.
In the second part of this section, we present some general observations concerning higgledy-piggledy line sets of minimal size and their duals, and state what it means for such sets to be \emph{optimal}.

\subsection{Construction methods}

There are several ways to construct (small) higgledy-piggledy sets of $k$-subspaces in $\pg{\n}{q}$.
Constructions via \emph{projection} or \emph{dualisation} make use of the existence of other higgledy-piggledy sets to construct new ones of similar size.
Construction via \emph{field reduction} relies on the existing knowledge of \emph{$\FF_q$-linear sets} to prove the existence of higgledy-piggledy sets of $k$-subspaces contained in Desarguesian spreads.

\subsubsection*{Construction by projection}

\begin{prop}
    Let $k\in\set{0,1,\dots,\n-1}$ and suppose that $\Strong$ is a strong $k$-blocking set of $\pg{\n}{q}$.
    Take a hyperplane $\Sigma$ and a point $P\notin\Strong\cup\Sigma$.
    Then $\Strong':=\sett{\vspan{P,S}\cap\Sigma}{S\in\Strong}$ is a strong $k$-blocking set of $\Sigma\cong\pg{\n-1}{q}$.
\end{prop}
\begin{proof}
    Suppose, to the contrary, that there exists an $(\n-k-1)$-subspace $\Pi\subseteq\Sigma$ that meets $\Strong'$ in a point set contained in an $(\n-k-2)$-space $\Pi'$.
    By definition of $\Strong'$, this means that $\vspan{\Pi,P}$ is an $(\n-k)$-space that meets $\Strong$ in a point set contained in the $(\n-k-1)$-space $\vspan{\Pi',P}$, a contradiction.
\end{proof}

\begin{crl}\label{Crl_Projection}
    Let $k\in\set{0,1,\dots,\n-1}$ and suppose that $\mathcal{K}$ is a higgledy-piggledy set of $k$-subspaces in $\pg{\n}{q}$.
    Take a hyperplane $\Sigma$ and a point $P\notin\Sigma$ not contained in any of the elements of $\mathcal{K}$.
    Then $\mathcal{K}':=\sett{\vspan{P,\kappa}\cap\Sigma}{\kappa\in\mathcal{K}}$ is a higgledy-piggledy set of $k$-subspaces in $\Sigma\cong\pg{\n-1}{q}$ of size at most $|\mathcal{K}|$.
\end{crl}

Corollary \ref{Crl_Projection} depicts the construction technique of higgledy-piggledy sets by \emph{projection}, which is a simple but potentially powerful technique (see e.g.\ Remark \ref{Rmk_EvenToOddProjection}).

\subsubsection*{Construction by dualisation}

A second construction technique arises by making use of the natural \emph{duality} of $\pg{\n}{q}$, and although this insight is anything but groundbreaking, we want to note that this has also been pointed out in \cite[Theorem $9$, Proposition $10$]{FancsaliSziklai3}.

\begin{prop}\label{Prop_Duality}
    Let $k\in\set{0,1,\dots,\n-1}$ and suppose that $\mathcal{K}$ is a higgledy-piggledy set of $k$-subspaces in $\pg{\n}{q}$ with $|\mathcal{K}|\leq q$.
    Then the set $\mathcal{K}^\bot$ consisting of the dual subspaces of the elements in $\mathcal{K}$ is a higgledy-piggledy set of $(\n-k-1)$-subspaces in $\pg{\n}{q}$.
\end{prop}
\begin{proof}
    By Theorem \ref{Res_CharHigPig}, no $(\n-k-1)$-subspace meets each element of $\mathcal{K}$.
    Taking the dual of this statement, we obtain the knowledge that no $k$-subspace meets each element of $\mathcal{K}^\bot$.
    Applying Theorem \ref{Res_CharHigPig} yet again, we conclude that $\mathcal{K}^\bot$ must be a higgledy-piggledy set of $(\n-k-1)$-subspaces in $\pg{\n}{q}$.
\end{proof}

Theorem \ref{Thm_LowerBound} is an excellent example of the usage of this method.
Moreover, as we will observe in Section \ref{Sect_HigPig4}, this technique will imply the existence of a small higgledy-piggledy plane set in $\pg{4}{q}$ (Corollary \ref{Crl_SixHigPigPlanes4}).

\subsubsection*{Construction by field reduction}

This particular method for constructing higgledy-piggledy sets is useful if (and only if) $\n+1$ is composite.

Let $\n',k\in\set{0,1,\dots,\n-1}$.
The idea behind \emph{field reduction} is interpreting a projective geometry $\pg{\n'}{q^{k+1}}$ as its underlying vector space $\V{\n'+1}{q^{k+1}}$, which is known to be isomorphic to $\V{(\n'+1)(k+1)}{q}$, which in turn naturally translates to $\pg{(\n'+1)(k+1)-1}{q}$.
In this way, one obtains a correspondence between subspaces of $\pg{\n'}{q^{k+1}}$ and subspaces of $\pg{(\n'+1)(k+1)-1}{q}$ by `reducing' the underlying field.
A great survey on this topic can be found in \cite{LavrauwVandeVoorde2}.

The authors of this work formally introduce the \emph{field reduction map}
\begin{equation}\label{Eq_FieldRed}
    \mathcal{F}_{\n'+1,k+1,q}:\pg{\n'}{q^{k+1}}\rightarrow\pg{(\n'+1)(k+1)-1}{q}\textnormal{,}
\end{equation}
which maps subspaces onto subspaces by viewing these as embedded projective geometries and applying field reduction.

\begin{df}
    A \emph{$k$-spread} of $\pg{\n}{q}$ is a set of pairwise disjoint $k$-spaces covering all points of $\pg{\n}{q}$.
\end{df}

For any point set $\mathcal{P}$ of $\pg{\n'}{q^{k+1}}$, define
\[
    \mathcal{F}_{\n'+1,k+1,q}(\mathcal{P}):=\sett{\mathcal{F}_{\n'+1,k+1,q}(P)}{P\in\mathcal{P}}\textnormal{.}
\]
One of the many properties of the field reduction map is the fact that if $\mathcal{P}$ is the set of all points in $\pg{\n'}{q^{k+1}}$, then $\mathcal{F}_{\n'+1,k+1,q}(\mathcal{P})$ is a $k$-spread of $\pg{(\n'+1)(k+1)-1}{q}$.
A $k$-spread isomorphic to this spread is generally called \emph{Desarguesian}; we will denote this spread by $\mathcal{D}_{\n'+1,k+1,q}$.

The following definition is a generalisation of the concept of subgeometries in a projective space.

\begin{df}
    Let $\mathcal{P}$ be a point set of $\pg{\n'}{q^{k+1}}$.
    Then we will call $\mathcal{P}$ an \emph{$\FF_q$-linear set} of rank $k'$, $k'\in\NN$, if there exists a $(k'-1)$-subspace $\kappa$ of $\pg{(\n'+1)(k+1)-1}{q}$ such that $\mathcal{F}_{\n'+1,k+1,q}(\mathcal{P})$ is precisely the set of all elements of $\mathcal{D}_{\n'+1,k+1,q}$ that intersect $\kappa$, and $k'$ is minimal w.r.t.\ this property.
\end{df}

Note that whenever an $\FF_q$-linear set in $\pg{\n'}{q^{k+1}}$ has a rank larger than $\n-k$, it contains all points of $\pg{\n'}{q^{k+1}}$.
With this in mind, the following theorem basically states that any point set that is not contained in a `proper' $\FF_q$-linear set gives rise to a higgledy-piggledy set of $k$-spaces.

\begin{thm}\label{Thm_FieldRed}
    Let $k\in\set{0,1,\dots,\n-1}$ such that $\n+1=(\n'+1)(k+1)$ for a certain $\n'\in\set{0,1,\dots,\n-1}$.
    Suppose that $\mathcal{P}$ is a point set of $\pg{\n'}{q^{k+1}}$ that is not contained in any $\FF_q$-linear set of rank at most $\n-k$.
    Then $\mathcal{F}_{\n'+1,k+1,q}(\mathcal{P})$ is a higgledy-piggledy set of pairwise disjoint $k$-subspaces in $\pg{\n}{q}$.
\end{thm}
\begin{proof}
    Suppose, to the contrary, that $\mathcal{F}_{\n'+1,k+1,q}(\mathcal{P})$ is not a higgledy-piggledy set of $k$-subspaces in $\pg{\n}{q}$.
    By Theorem \ref{Res_CharHigPig}, there exists an $(\n-k-1)$-subspace that meets all elements of $\mathcal{F}_{\n'+1,k+1,q}(\mathcal{P})$, implying that the latter is contained in an $\FF_q$-linear set of rank at most $\n-k$, a contradiction.
\end{proof}

The idea behind Theorem \ref{Thm_FieldRed} is to search for higgledy-piggledy sets as a subset of a Desarguesian spread and was first cleverly used in \cite{BartoliCossidenteMarinoPavese} for the case $\n'=2$ and $k=1$ to prove the existence of seven lines of $\pg{5}{q}$ in higgledy-piggledy arrangement (see Theorem \ref{Res_SporadicExamples}($3.$)).
As a side note, using this method, one can even produce a slightly more elegant proof for Theorem \ref{Res_SporadicExamples}($1.$) than the one currently available in the literature \cite[Theorem $3.7$, Example $9$]{DavydovEtAl,FancsaliSziklai2}: as a Baer subline (see Definition \ref{Def_Subline} for $m=2$) is uniquely determined by any three of its points, one can use Theorem \ref{Thm_FieldRed} and choose four points in $\pg{1}{q^2}$ not contained in a Baer subline (which is precisely an $\FF_q$-linear set of rank $2$) to obtain a higgledy-piggledy set of four pairwise disjoint lines of $\pg{3}{q}$.

\subsection{Optimal higgledy-piggledy line sets and sets of \texorpdfstring{$\boldsymbol{(\n-2)}$}{(N-2)}-subspaces}

We now shift our focus to general higgledy-piggledy line sets of smallest theoretical size.

\begin{lm}\label{Lm_SpanProperty}
    Suppose that $\mathcal{L}$ is a higgledy-piggledy line set of $\pg{\n}{q}$ with $|\mathcal{L}|=\n+\left\lfloor\frac{\n}{2}\right\rfloor\leq q$.
    Then every $\left\lceil\frac{\n+1}{2}\right\rceil$ lines of $\mathcal{L}$ span the whole space $\pg{\n}{q}$.
\end{lm}
\begin{proof}
    Suppose, to the contrary, that there exists a subset $\mathcal{L}'\subseteq\mathcal{L}$ consisting of $\left\lceil\frac{\n+1}{2}\right\rceil$ lines contained in a fixed hyperplane $\Sigma$.
    For each $\ell\in\mathcal{L}\setminus\mathcal{L}'$, choose a point in the non-empty subspace $\Sigma\cap\ell$.
    This results in a choice of at most $\n+\left\lfloor\frac{\n}{2}\right\rfloor-\left\lceil\frac{\n+1}{2}\right\rceil=\n-1$ points in $\Sigma$ spanning a subspace $\Pi\subseteq\Sigma$ of dimension at most $\n-2$.
    Any $(\n-2)$-subspace of $\Sigma$ through $\Pi$ is an $(\n-2)$-subspace that intersects every line of $\mathcal{L}$, contradicting Theorem \ref{Res_CharHigPig}.
\end{proof}

In the propositions below, a \emph{pair of lines} is meant to be an \emph{unordered} pair of lines, i.e.\ a set of two lines.

\begin{prop}\label{Prop_OptimalHigPigLines}
    Suppose that $\mathcal{L}$ is a higgledy-piggledy line set of $\pg{\n}{q}$ with $|\mathcal{L}|=\n+\left\lfloor\frac{\n}{2}\right\rfloor\leq q$.
    Then the following holds.
    \begin{enumerate}
        \item If $\n$ is odd, the lines of $\mathcal{L}$ are pairwise disjoint.
        \item If $\n\geq4$ is even, at most two lines of $\mathcal{L}$ intersect.
    \end{enumerate}
\end{prop}
\begin{proof}
    Let $\n$ be odd.
    The statement is trivial when $\n=1$, hence we can assume that $\n\geq3$.
    Suppose, to the contrary, that two lines $\ell,\ell'\in\mathcal{L}$ span a plane $\pi$.
    Consider $m:=\frac{\n-3}{2}$ lines $\ell_1,\ell_2,\dots,\ell_m\in\mathcal{L}\setminus\set{\ell,\ell'}$.
    Then $\vspan{\ell,\ell',\ell_1,\ell_2,\dots,\ell_m}=\vspan{\pi,\ell_1,\ell_2,\dots,\ell_m}$ is a span of $\frac{\n+1}{2}$ lines of $\mathcal{L}$ equal to a space of dimension at most $\n-1$, contradicting Lemma \ref{Lm_SpanProperty}.
    
    Let $\n\geq4$ be even.
    Suppose, to the contrary, that there exist two pairs of intersecting lines with corresponding intersection points $S_1$ and $S_2$; define $\mathcal{L}'$ to be the set of these lines.
    We distinguish two cases depending on the size of $\mathcal{L}'$ and equality of the intersection points $S_1$ and $S_2$.
    If $|\mathcal{L}'|=3$ or if $S_1=S_2$, then there exists a solid $\sigma$ containing at least three lines of $\mathcal{L}'$.
    Consider $m:=\frac{\n-4}{2}$ lines $\ell_1,\ell_2,\dots,\ell_m\in\mathcal{L}\setminus\mathcal{L}'$.
    Then $\vspan{\sigma,\ell_1,\ell_2,\dots,\ell_m}$ is a space of dimension at most $\n-1$ that contains at least $\frac{\n+2}{2}$ lines of $\mathcal{L}$, contradicting Lemma \ref{Lm_SpanProperty}.
    If $|\mathcal{L}'|=4$ and $S_1\neq S_2$, then the line $s:=\vspan{S_1,S_2}$ is well-defined and intersects all four lines of $\mathcal{L}'$.
    Consider $m:=\frac{\n-2}{2}$ lines $\ell_1,\ell_2,\dots,\ell_m\in\mathcal{L}\setminus\mathcal{L}'$.
    As $\vspan{s,\ell_1,\ell_2,\dots,\ell_m}$ has dimension at most $\n-1$, we can choose a hyperplane $\Sigma$ through this space.
    For each $\ell\in\mathcal{L}\setminus\left(\mathcal{L}'\cup\set{\ell_1,\ell_2,\dots,\ell_m}\right)$, choose a point in the non-empty subspace $\Sigma\cap\ell$.
    This results in a choice of at most $\n+\frac{\n}{2}-(4+m)=\n-3$ points in $\Sigma$ spanning, together with the line $s$, a subspace $\Pi\subseteq\Sigma$ of dimension at most $\n-2$.
    Any $(\n-2)$-subspace of $\Sigma$ through $\Pi$ is an $(\n-2)$-subspace that intersects every line of $\mathcal{L}$, contradicting Theorem \ref{Res_CharHigPig}.
\end{proof}

\begin{prop}\label{Prop_OptimalHigPigN-2}
    Suppose that $\mathcal{K}$ is a higgledy-piggledy set of subspaces in $\pg{\n}{q}$ of dimension $\n-2$, with $|\mathcal{K}|=\n+\left\lfloor\frac{\n}{2}\right\rfloor\leq q$.
    Then every $\left\lceil\frac{\n+1}{2}\right\rceil$ elements of $\mathcal{K}$ have no point in common.
    Moreover, the following holds.
    \begin{enumerate}
        \item If $\n\geq3$ is odd, the elements of $\mathcal{K}$ pairwise intersect in an $(\n-4)$-subspace.
        \item If $\n\geq4$ is even, at most two elements of $\mathcal{K}$ intersect in an $(\n-3)$-subspace.
    \end{enumerate}
\end{prop}
\begin{proof}
    These results follow immediately by combining Proposition \ref{Prop_Duality} with Lemma \ref{Lm_SpanProperty} and Proposition \ref{Prop_OptimalHigPigLines}, respectively.
\end{proof}

Note that, alternatively, we could have dualised both statement and proof of Lemma \ref{Lm_SpanProperty} and Proposition \ref{Prop_OptimalHigPigLines} to obtain Proposition \ref{Prop_OptimalHigPigN-2}.

\bigskip
Propositions \ref{Prop_OptimalHigPigLines} and \ref{Prop_OptimalHigPigN-2} give us an understanding of the smallest possible set-ups for higgledy-piggledy sets of $k$-subspaces, $k\in\set{1,\n-2}$, of size at most $q$.
Therefore, we define the following accordingly.

\begin{df}
    Let $\n\geq3$ and $k\in\set{1,\n-2}$.
    Suppose that $\mathcal{K}$ is a higgledy-piggledy set of $k$-subspaces in $\pg{\n}{q}$ with $|\mathcal{K}|=\n+\left\lfloor\frac{\n}{2}\right\rfloor\leq q$.
    Then we will call $\mathcal{K}$ \emph{optimal} if
    \begin{enumerate}
        \item either $\n$ is odd, or
        \item $\n$ is even and two elements of $\mathcal{K}$ intersect in a $(k-1)$-subspace.
    \end{enumerate}
\end{df}

\begin{rmk}\label{Rmk_EvenToOddProjection}
    Let $\n\geq3$, $k\in\set{1,\n-2}$, and suppose there exists an optimal higgledy-piggledy set of $k$-subspaces for each odd, respectively even, $\n$.
    Then, with the aid of projection (Corollary \ref{Crl_Projection}), one can prove the existence of a higgledy-piggledy set of $k$-subspaces of size $\n+\left\lfloor\frac{\n}{2}\right\rfloor+1$ for each even, respectively odd, $\n$ (for $\n$ even, one simply has to choose the point of projection within the span of the two $k$-subspaces that maximally intersect).
    This reduces the search for small higgledy-piggledy sets of $k$-subspaces, $k\in\set{1,\n-2}$, to one parity class of $\n$.
\end{rmk}

\section{Higgledy-piggledy sets of \texorpdfstring{$\boldsymbol{\pgTitle{4}{q}}$}{PG(4,q)}}\label{Sect_HigPig4}

This section aims to prove the existence of an optimal higgledy-piggledy set of $k$-subspaces in $\pg{4}{q}$, $k\in\set{1,2}$.

\bigskip
We defined a blocking set of $\pg{2}{q}$ to be a point set meeting every line of the projective plane (see Definition \ref{Def_BlockingSet}).
In the literature, researchers also investigated point sets of $\pg{2}{q}$ that meet every line of a fixed line set.
In particular, blocking sets w.r.t.\ the external lines to an irreducible conic were considered.
In $2006$, Aguglia and Korchm\'{a}ros \cite{AgugliaKorchmaros} managed to characterise such blocking sets of minimal size in case $q$ is odd.
One year later, Giulietti \cite{Giulietti} tackled the case of $q$ even.
Although a full characterisation is known, for the purpose of this section, we only require the following result.

\begin{thm}[{\cite[Theorem $1.1$]{AgugliaKorchmaros,Giulietti}}]\label{Res_BlockingExternalLines}
    The minimum size of a blocking set w.r.t.\ external lines to an irreducible conic of $\pg{2}{q}$ is $q-1$.
\end{thm}

Throughout this section, keep the following base configuration in mind.

\begin{conf}\label{Conf_FourLines}
    Suppose $\Sigma_1$, $\Sigma_2$ and $\Sigma_3$ are solids of $\pg{4}{q}$ such that their intersection $m:=\Sigma_1\cap\Sigma_2\cap\Sigma_3$ is a line; let $M_1$ and $M_2$ be two distinct points on $m$.
    Define, for every $i,j\in\set{1,2,3}$, $i<j$, the plane $\pi_{ij}:=\Sigma_i\cap\Sigma_j$ and let $P_{ij}\in\pi_{ij}\setminus m$ be a point.
    Consider, for each $i\in\set{1,2}$, the lines $\ell_{i2}:=\vspan{P_{12},P_{i3}}$ and $\ell_{i1}$ in $\Sigma_i$ through $M_i$ not intersecting $\ell_{i2}$ and not contained in $\pi_{12}$ or $\pi_{i3}$.
    Define the line $s:=\vspan{P_{13},P_{23}}$, the plane $\beta:=\vspan{\ell_{11},\ell_{21}}\cap\Sigma_3$ and their intersection point $S:=s\cap\beta$.
    To conclude, consider the following projections:
    \begin{enumerate}
        \item the line $\ell_{11}':=\vspan{P_{13},\ell_{11}}\cap\pi_{12}$, and
        \item the line $\ell_{i1}'':=\vspan{P_{12},\ell_{i1}}\cap\pi_{i3}$ for each $i\in\set{1,2}$.
    \end{enumerate}
    See Figure \ref{Fig_HigPigDim4} for a visualisation of this configuration, where the lines $\ell_{11}$, $\ell_{12}$, $\ell_{21}$ and $\ell_{22}$ (and $\ell_{31}$, see Configuration \ref{Conf_FiveLines}) are drawn in \textcolor{red}{red}, while their projections as defined above are shown in \textcolor{orange}{orange}.
\end{conf}

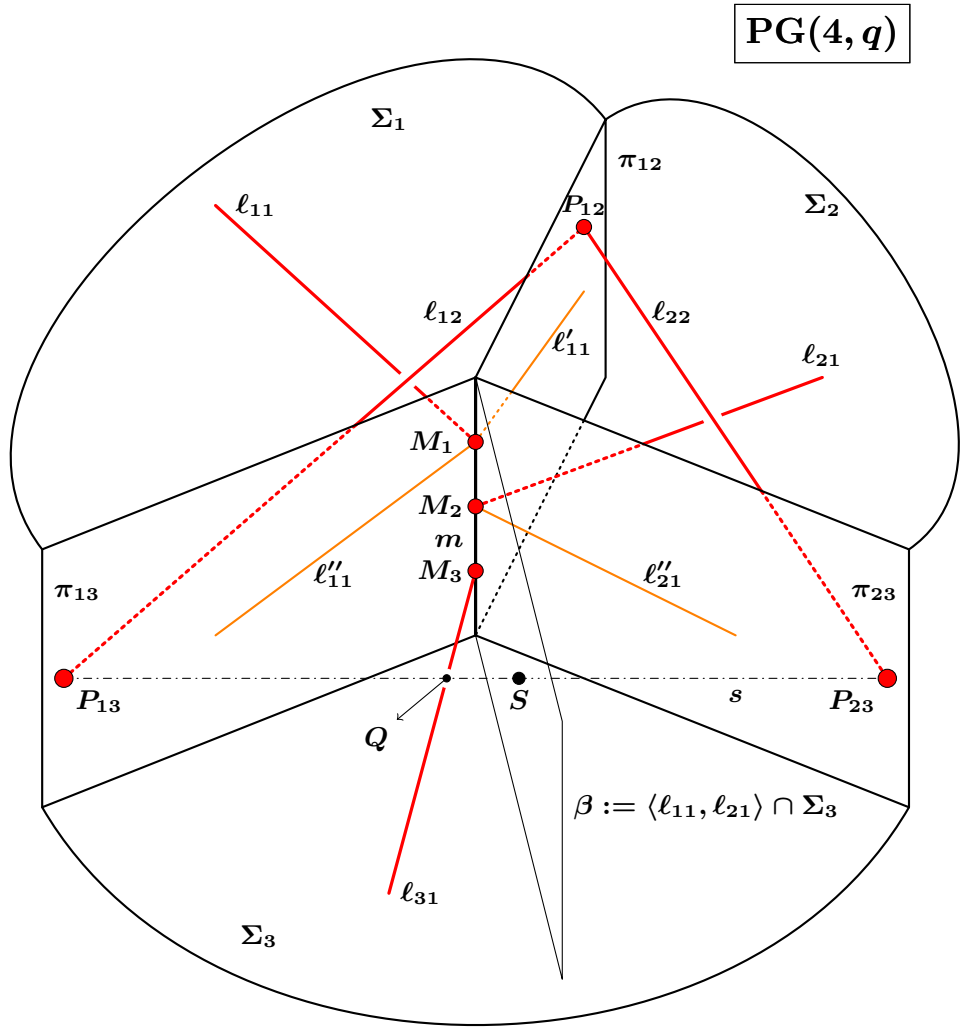
\begin{figure}
    \begin{center}\begin{tikzpicture}[scale=1.14, every node/.style={scale=1.14}]
        \path[name path=l11] (-45/79,219/79) -- (-3,5);
        \path[name path=l12] (-105/76,93/38) -- (7/12,25/6);
        \path[name path=l21] (60/31,69/31) -- (4,3);
        \path[name path=l22] (1.25,4.75) -- (145/44,37/22);
        \path[name path=l31] (0,0.75) -- (-1,-3);
        \path[name path=s] (-4.75,-0.5) -- (0.5,-0.5);
        
        \node[draw,fill=none] at (4,7) {\large\textbf{PG}$\boldsymbol{(4,q)}$};
		\node[draw,fill=none] at (4,7) {\large\textbf{PG}$\boldsymbol{(4,q)}$}; 
		
		\draw[thick, line join=round, line cap=round] (1.5,6) to [out=180,in=90,bend right=90] (-5,1);
        \node[draw=none, fill=none] at (-1,6) {\footnotesize$\boldsymbol{\Sigma_1}$};
        \draw[thick, line join=round, line cap=round] (1.5,6) to [out=0,in=90,bend left=90] (5,1);
        \node[draw=none, fill=none] at (4,5) {\footnotesize$\boldsymbol{\Sigma_2}$};
        \draw[thick, line join=round, line cap=round] (-5,-2) to [out=270,in=270,bend right=60] (5,-2);
        \node[draw=none, fill=none] at (-2.5,-3.5) {\footnotesize$\boldsymbol{\Sigma_3}$};
        
        \draw[very thick, dotted, line join=round, line cap=round, color=red] (0,2.25) -- (-45/79,219/79);
        \draw[very thick, line join=round, line cap=round, color=red] (-45/79,219/79) -- (-3,5);
        \node[draw=none, fill=none, anchor=west] at (-2.9,5) {\footnotesize$\boldsymbol{\ell_{11}}$};
        \draw[color=white, fill=white, name intersections={of=l11 and l12}] (intersection-1) circle (3pt);
        
        \draw[very thick, dotted, line join=round, line cap=round, color=red] (-4.75,-0.5) -- (-105/76,93/38);
        \draw[very thick, line join=round, line cap=round, color=red] (-105/76,93/38) -- (7/12,25/6);
        \draw[very thick, dotted, line join=round, line cap=round, color=red] (7/12,25/6) -- (1.25,4.75);
        \node[draw=none, fill=none, anchor=south east] at (0,3.5) {\footnotesize$\boldsymbol{\ell_{12}}$};
        
        \draw[thick, line join=round, line cap=round] (0,0) -- (0,3) -- (1.5,6) -- (1.5,3) -- (1.25,2.5);
        \draw[thick, dotted, line join=round, line cap=round] (1.25,2.5) -- (0,0);
        \node[draw=none, fill=none, anchor=west] at (1.5,5.5) {\footnotesize$\boldsymbol{\pi_{12}}$};
        \draw[thick, dotted, line join=round, line cap=round, color=orange] (0,2.25) -- (5/12,17/6);
        \draw[thick, line join=round, line cap=round, color=orange] (5/12,17/6) -- (1.25,4);
        \node[draw=none, fill=none, anchor=north west] at (0.75,3.7) {\footnotesize$\boldsymbol{\ell_{11}'}$};
        
        \draw[very thick, dotted, line join=round, line cap=round, red] (0,1.5) -- (60/31,69/31);
        \draw[very thick, line join=round, line cap=round, color=red] (60/31,69/31) -- (4,3);
        \node[draw=none, fill=none, anchor=south] at (4,3) {\footnotesize$\boldsymbol{\ell_{21}}$};
        \draw[color=white, fill=white, name intersections={of=l21 and l22}] (intersection-1) circle (3pt);
        
        \draw[very thick, line join=round, line cap=round, color=red] (1.25,4.75) -- (145/44,37/22);
        \draw[very thick, dotted, line join=round, line cap=round, color=red] (145/44,37/22) -- (4.75,-0.5);
        \node[draw=none, fill=none, anchor=south west] at (1.9,3.5) {\footnotesize$\boldsymbol{\ell_{22}}$};
        
        \draw[fill=red] (1.25,4.75) circle (2.5pt);
        \node[draw=none, fill=none, anchor=south] at (1.25,4.75) {\scriptsize$\boldsymbol{P_{12}}$};
        
        \draw[thick, line join=round, line cap=round] (0,0) -- (0,3) -- (-5,1) -- (-5,-2) -- cycle;
        \node[draw=none, fill=none, anchor=west] at (-5,0.5) {\footnotesize$\boldsymbol{\pi_{13}}$};
        \draw[thick, line join=round, line cap=round, color=orange] (0,2.25) -- (-3,0);
        \node[draw=none, fill=none, anchor=north west] at (-2,1) {\footnotesize$\boldsymbol{\ell_{11}''}$};
        
        \draw[thick, line join=round, line cap=round] (0,0) -- (0,3) -- (5,1) -- (5,-2) -- cycle;
        \node[draw=none, fill=none, anchor=east] at (5,0.5) {\footnotesize$\boldsymbol{\pi_{23}}$};
        \draw[thick, line join=round, line cap=round, color=orange] (0,1.5) -- (3,0);
        \node[draw=none, fill=none, anchor=north west] at (1.8,1) {\footnotesize$\boldsymbol{\ell_{21}''}$};
        
        \draw[very thick, line join=round, line cap=round] (0,0) -- (0,3);
        \node[draw=none, fill=none, anchor=east] at (0,1.1) {\footnotesize$\boldsymbol{m}$};
        
        \draw[very thick, line join=round, line cap=round, color=red] (0,0.75) -- (-1,-3);
        \node[draw=none, fill=none, anchor=west] at (-1,-3) {\footnotesize$\boldsymbol{\ell_{31}}$};
        \draw[color=white, fill=white, name intersections={of=l31 and s}] (intersection-1) circle (3pt);
        
        \draw[dashdotted, line join=round, line cap=round] (-4.75,-0.5) -- (0.5,-0.5);
        \draw[dotted, line join=round, line cap=round] (0.5,-0.5) -- (7/8,-0.5);
        \draw[dashdotted, line join=round, line cap=round] (7/8,-0.5) -- (4.75,-0.5);
        \node[draw=none, fill=none, anchor=north] at (3,-0.5) {\footnotesize$\boldsymbol{s}$};
        
        \draw[line join=round, line cap=round] (0,0) -- (0,3) -- (1,-1) -- (1,-4) -- cycle;
        \node[draw=none, fill=none, anchor=west] at (1,-2) {\footnotesize$\boldsymbol{\beta:=\vspan{\ell_{11},\ell_{21}}\cap\Sigma_3}$};
        \draw[fill=black, name intersections={of=l31 and s}] (intersection-1) circle (1.2pt);
        \draw[->, line join=round, line cap=round, name intersections={of=l31 and s}] (intersection-1) -- +(220:0.75);
        \node[draw=none, fill=none, name intersections={of=l31 and s}] at ($(intersection-1)+(220:1)$) {\footnotesize$\boldsymbol{Q}$};
        \draw[fill=black] (0.5,-0.5) circle (2pt);
        \node[draw=none, fill=none, anchor=north] at (0.5,-0.5) {\footnotesize$\boldsymbol{S}$};
        
        \draw[fill=red] (0,2.25) circle (2.5pt);
        \node[draw=none, fill=none, anchor=east] at (-0.1,2.25) {\footnotesize$\boldsymbol{M_1}$};
        \draw[fill=red] (0,1.5) circle (2.5pt);
        \node[draw=none, fill=none, anchor=east] at (0,1.5) {\footnotesize$\boldsymbol{M_2}$};
        \draw[fill=red] (0,0.75) circle (2.5pt);
        \node[draw=none, fill=none, anchor=east] at (0,0.75) {\footnotesize$\boldsymbol{M_3}$};
        
        \draw[fill=red] (-4.75,-0.5) circle (3pt);
        \node[draw=none, fill=none, anchor=north west] at (-4.75,-0.5) {\footnotesize$\boldsymbol{P_{13}}$};
        \draw[fill=red] (4.75,-0.5) circle (3pt);
        \node[draw=none, fill=none, anchor=north east] at (4.75,-0.5) {\footnotesize$\boldsymbol{P_{23}}$};
    \end{tikzpicture}\end{center}
    \caption{The set-up as described in Configuration \ref{Conf_FourLines} and Configuration \ref{Conf_FiveLines}.}
    \label{Fig_HigPigDim4}
\end{figure}

\begin{nt}\label{Nt_Planes4}
    Denote with $\Pi^{(4)}$ the set of all planes of $\pg{4}{q}$ that intersect each of the lines $\ell_{11}$, $\ell_{12}$, $\ell_{21}$ and $\ell_{22}$.
\end{nt}

\begin{lm}\label{Lm_CharPlanesPi4}
    Each plane of $\Pi^{(4)}$ either
    \begin{enumerate}
        \item intersects $\Sigma_3$ in a line of $\pi_{13}$ through $M_2$ not equal to $m$,
        \item intersects $\Sigma_3$ in a line of $\pi_{23}$ through $M_1$ not equal to $m$,
        \item is equal to $\pi_{12}$, or
        \item intersects $\pi_{12}$ in precisely one point not contained in $\vspan{M_i,P_{12}}\setminus\set{M_i,P_{12}}$, $i\in\{1,2\}$.
    \end{enumerate}
    Moreover, for every point $A\in\pi_{12}\setminus\big(\vspan{M_1,P_{12}}\cup\vspan{M_2,P_{12}}\big)$, there exists a unique plane of $\Pi^{(4)}$ that intersects $\pi_{12}$ in precisely the point $A$.
\end{lm}
\begin{proof}
    Consider a plane $\alpha\in\Pi^{(4)}$ and suppose that $\alpha$ is contained in $\Sigma_i$ for a certain $i\in\set{1,2}$.
    Then $\alpha$ has to contain the points $M_{3-i}$ and $P_{12}$ to be able to intersect the lines $\ell_{(3-i)1}$ and $\ell_{(3-i)2}$, respectively, and hence has to intersect $\Sigma_3$ in a line $r\subseteq\pi_{i3}$ through $M_{3-i}$.
    If $r$ differs from $m$, then either property $1.$ or property $2.$ is true.
    If $r=m$, then property $3.$ holds.
    
    Now suppose that $\alpha$ is not contained in $\Sigma_1$ or $\Sigma_2$, then $\alpha$ intersects these solids in lines $a_1$ and $a_2$, respectively.
    If $\alpha$ intersects $\pi_{12}$ in a line, then this line has to be equal to $a_1=a_2$ which consequently has to contain the non-collinear points $M_1$, $M_2$ and $P_{12}$ to be able to intersect the lines $\ell_{11}$, $\ell_{21}$, $\ell_{12}$ and $\ell_{22}$, a contradiction.
    Hence, $\alpha$ intersects $\pi_{12}$ in precisely a point $A=a_1\cap a_2$ and thus $\alpha=\vspan{a_1,a_2}$.
    It is clear that, for $i\in\set{1,2}$, the line $a_i$ has to intersect the disjoint lines $\ell_{i1}$ and $\ell_{i2}$.
    If $A\notin\set{M_1,M_2,P_{12}}$, then there exists a unique line through $A$ intersecting both these lines, which hence has to be equal to $a_i$.
    This means that $A$ cannot be contained in $\vspan{M_i,P_{12}}$, as else $a_i=\vspan{M_i,P_{12}}\subseteq\pi_{12}$, and that $\alpha$ is uniquely defined by its intersection point $A$ with $\pi_{12}$, finishing the proof.
\end{proof}

Given the above lemma, we can now introduce the following notation.

\begin{nt}
    For every point $A\in\pi_{12}\setminus\big(\vspan{M_1,P_{12}}\cup\vspan{M_2,P_{12}}\big)$, let $\alpha^{(A)}$ be the unique plane of $\Pi^{(4)}$ intersecting $\pi_{12}$ in precisely the point $A$.
    For every $i\in\set{1,2,3}$, define the line $a_i^{(A)}:=\alpha^{(A)}\cap\Sigma_i$.
\end{nt}

\begin{lm}\label{Lm_CoplanarAndConcurrent}
    Let $\mathfrak{a}$ be a line in $\pi_{12}$ through $P_{12}$, not equal to $\vspan{M_1,P_{12}}$ or $\vspan{M_2,P_{12}}$.
    Then $\sett{a_3^{(A)}}{A\in\mathfrak{a}\setminus\set{P_{12}}}$ is a set of $q$ lines lying in a plane of $\Sigma_3$ through $s$ and going through a fixed point of $\beta$.
\end{lm}
\begin{proof}
    For every $A\in\mathfrak{a}\setminus\set{P_{12}}$ and each $i\in\set{1,2}$, the line $a_i^{(A)}$ is contained in $\vspan{\mathfrak{a},\ell_{i2}}$, a plane independent of the choice of $A$ that intersects $\ell_{i1}$ necessarily in a point $Q_i\notin\big(\pi_{12}\cup\pi_{i3}\big)$.
    The line $a_i^{(A)}$ has to intersect $\ell_{i1}$, thus it has to go through the point $Q_i$.
    As a first result, all lines of $\sett{a_3^{(A)}}{A\in\mathfrak{a}\setminus\set{P_{12}}}$ lie in the plane $\vspan{\mathfrak{a},\ell_{12},\ell_{22}}\cap\Sigma_3$ and hence are coplanar; the corresponding plane contains both $P_{13}$ and $P_{23}$ and hence also the line $\vspan{P_{13},P_{23}}=s$.
    As a second result, all planes of $\sett{\alpha^{(A)}}{A\in\mathfrak{a}\setminus\set{P_{12}}}$ go through the line $\vspan{Q_1,Q_2}$, the latter necessarily intersects $\Sigma_3$ in a point $Q_3\notin\set{Q_1,Q_2}$.
    Consequently, all lines of $\sett{a_3^{(A)}}{A\in\mathfrak{a}\setminus\set{P_{12}}}$ have to go through the point $Q_3$.
    As $Q_1\in\ell_{11}$ and $Q_2\in\ell_{21}$, the line $\vspan{Q_1,Q_2}$ lies in $\vspan{\ell_{11},\ell_{21}}$ and, hence, $Q_3$ lies in $\vspan{\ell_{11},\ell_{21}}\cap\Sigma_3=\beta$.
\end{proof}

We can now introduce yet another notation.

\begin{nt}
    By Lemma \ref{Lm_CoplanarAndConcurrent}, we know that for every line $\mathfrak{a}$ in $\pi_{12}$ through $P_{12}$, not equal to $\vspan{M_1,P_{12}}$ or $\vspan{M_2,P_{12}}$, the $q$ lines of $\sett{a_3^{(A)}}{A\in\mathfrak{a}\setminus\set{P_{12}}}$ are coplanar and concurrent; we will denote this unique plane by $\gamma^{(\mathfrak{a})}\supseteq s$ and this unique point of concurrence by $\mathcal{A}^{(\mathfrak{a})}\in\beta$.
\end{nt}

\begin{lm}\label{Lm_IntersectionPointsAreConic}
    The point set $\sett{\mathcal{A}^{(\mathfrak{a})}}{P_{12}\in\mathfrak{a}\subseteq\pi_{12}, \mathfrak{a}\notin\set{\vspan{M_1,P_{12}},\vspan{M_2,P_{12}}}}\cup\set{M_1,M_2}$ is an irreducible conic contained in $\beta$ that contains the point $S$.
\end{lm}
\begin{proof}
    The fact that $\ell_{11}$ and $\ell_{12}$ are disjoint implies that $P_{12}\notin\ell_{11}'$, hence each point $A\in\ell_{11}'\setminus\big(\set{M_1}\cup\vspan{M_2,P_{12}}\big)$ defines a distinct line $\vspan{A,P_{12}}$.
    As a consequence, each of the points in $\sett{\mathcal{A}^{(\mathfrak{a})}}{P_{12}\in\mathfrak{a}\subseteq\pi_{12}, \mathfrak{a}\notin\set{\vspan{M_1,P_{12}},\vspan{M_2,P_{12}}}}$ corresponds to at least one of the $q-1$ points in $\ell_{11}'\setminus\big(\set{M_1}\cup\vspan{M_2,P_{12}}\big)$.
    By Lemma \ref{Lm_CoplanarAndConcurrent}, it suffices to prove the statement for the set of intersection points of the lines in $\sett{a_3^{(A)}}{A\in\ell_{11}'\setminus\big(\set{M_1}\cup\vspan{M_2,P_{12}}\big)}$ with the plane $\beta$.
    
    By definition of $\ell_{11}'$, all lines of $\sett{a_1^{(A)}}{A\in\ell_{11}'\setminus\big(\set{M_1}\cup\vspan{M_2,P_{12}}\big)}$ go through $P_{13}$, hence the lines of $\sett{a_3^{(A)}}{A\in\ell_{11}'\setminus\big(\set{M_1}\cup\vspan{M_2,P_{12}}\big)}$ go through $P_{13}$ as well.
    On the other hand, the lines $\ell_{11}'$, $\ell_{21}$ and $\ell_{22}$ are pairwise disjoint and lie in the solid $\Sigma_2$, hence these define a unique regulus $\mathcal{R}$ corresponding to a hyperbolic quadric $\mathcal{Q}$; let $\mathcal{R}'$ denote its opposite regulus.
    As the lines of $\sett{a_2^{(A)}}{A\in\ell_{11}'\setminus\big(\set{M_1}\cup\vspan{M_2,P_{12}}\big)}$ each have to intersect $\ell_{11}'$, $\ell_{21}$ and $\ell_{22}$, these lines are contained in $\mathcal{R}'$.
    
    We claim that $\mathcal{Q}\cap\pi_{23}$ is an irreducible conic.
    To prove this, first observe that $\ell_{11}'$ intersects the line $\vspan{M_2,P_{12}}$ in a point other than $M_2$ or $P_{12}$.
    As $M_2$ and $P_{12}$ are contained in $\mathcal{Q}$, this implies that $\vspan{M_2,P_{12}}$ is a generator of $\mathcal{Q}$.
    Hence, $M_2$ is contained in the following two generators of $\mathcal{Q}$: $\vspan{M_2,P_{12}}$ and $\ell_{21}$, neither of which are contained in $\pi_{23}$.
    As a consequence, there does not exist a generator of $\mathcal{Q}$ in $\pi_{23}$ through $M_2\in\mathcal{Q}$, which implies that $\mathcal{Q}\cap\pi_{23}$ is an irreducible conic $\mathcal{C}$ (containing $M_1$, $M_2$ and $P_{23}$).
    
    In conclusion, each of the $q-1$ lines of $\sett{a_3^{(A)}}{A\in\ell_{11}'\setminus\big(\set{M_1}\cup\vspan{M_2,P_{12}}\big)}$ intersects the plane $\pi_{13}$ in the point $P_{13}$ and intersects the plane $\pi_{23}$ in a distinct point of $\mathcal{C}\setminus\set{M_1,M_2}$; hence, these lines are generators of the cone with vertex $P_{13}$ and base $\mathcal{C}$.
    Switching our views to the plane $\beta$ instead of the plane $\pi_{23}$ simply switches the base of this cone and hence finishes the proof.
\end{proof}

Having obtained the above lemma, we can yet again announce a notation.

\begin{nt}
    For every line $\mathfrak{a}$ in $\pi_{12}$ through $P_{12}$, not equal to $\vspan{M_1,P_{12}}$ or $\vspan{M_2,P_{12}}$, let $r^{(\mathfrak{a})}$ be the unique line in $\gamma^{(\mathfrak{a})}$ through $\mathcal{A}^{(\mathfrak{a})}$ not contained in $\sett{a_3^{(A)}}{A\in\mathfrak{a}\setminus\set{P_{12}}}$; note that such a line is skew to $m$ and is never equal to $s$.
\end{nt}

We are now ready to choose a fifth line $\ell_{31}$ that is skew to most planes of $\Pi^{(4)}$.

\begin{conf}\label{Conf_FiveLines}
    Let $q\neq2$; we extend Configuration \ref{Conf_FourLines}.
    Let $t$ be the tangent line through $S$ w.r.t.\ the irreducible conic described in Lemma \ref{Lm_IntersectionPointsAreConic}, let $M_0:=t\cap m\notin\set{M_1,M_2}$ and consider the line $\mathfrak{a}_0:=\vspan{M_0,P_{12}}\subseteq\pi_{12}$; note that $\mathcal{A}^{(\mathfrak{a}_0)}=S$, as all lines of its corresponding bundle have to intersect $\beta$ in a point of the conic lying on the tangent line $t$ (which is part of this bundle).
    Choose a point $M_3\in m\setminus\set{M_0,M_1,M_2}$ and choose $\ell_{31}$ to be a line through $M_3$ intersecting $r^{(\mathfrak{a}_0)}$ in a point outside of $\pi_{13}\cup\pi_{23}\cup\beta$.
    Note that, in this way, $\ell_{31}$ is skew to all $q$ lines of $\sett{a_3^{(A)}}{A\in\mathfrak{a}_0\setminus\set{P_{12}}}$, in particular the line $s$.
    Finally, define $Q:=\vspan{m,\ell_{31}}\cap s$.
    
    Be sure to keep Figure \ref{Fig_HigPigDim4} at hand to maintain an overview of this configuration.
\end{conf}

\begin{nt}\label{Nt_Planes5}
    Denote with $\Pi^{(5)}$ the set of all planes of $\Pi^{(4)}$ that intersect $\ell_{31}$.
\end{nt}

\begin{thm}\label{Thm_SixHigPigLines4}
    There exist six lines in $\pg{4}{q}$ in higgledy-piggledy arrangement, two of which intersect.
\end{thm}
\begin{proof}
    One can easily check the statement for $q=2$ using, for example, the package FinInG within GAP\footnote{The authors of \cite{AlfaranoBorelloNeriRavagnani} showed that the smallest strong blocking set obtainable in $\pg{4}{2}$ has size $13$.}.
    Therefore, we can assume that $q\neq2$ throughout this proof and consider Configuration \ref{Conf_FiveLines}.
    By Theorem \ref{Res_CharHigPig}, it suffices to prove that there exists a sixth line $\ell_{32}$ skew to all planes of $\Pi^{(5)}$.
    Considering the four properties described in Lemma \ref{Lm_CharPlanesPi4}, all planes of $\Pi^{(5)}$ either meet property $3.$ or $4.$ due to the choice of $\ell_{31}$ (see Configuration \ref{Conf_FiveLines}).
    Hence, we can consider a partition $\set{\Pi_1,\Pi_2,\Pi_3,\Pi_4}$ of $\Pi^{(5)}$, where
    \begin{itemize}
        \item $\Pi_1$ is the set of all planes of $\Pi^{(5)}$ intersecting the plane $\pi_{12}$ in precisely a point not contained in $\vspan{M_1,P_{12}}\cup\vspan{M_2,P_{12}}$,
        \item $\Pi_2$ is the set of all planes of $\Pi^{(5)}$ intersecting the plane $\pi_{12}$ in precisely the point $P_{12}$,
        \item $\Pi_3$ is the set of all planes of $\Pi^{(5)}$ intersecting the plane $\pi_{12}$ in precisely a point of $\set{M_1,M_2}$, and
        \item $\Pi_4:=\set{\pi_{12}}$.
    \end{itemize}
    By Lemma \ref{Lm_CoplanarAndConcurrent}, the planes of $\Pi_1$ intersect the solid $\Sigma_3$ in a set of $q^2-q$ lines, grouped in $q-1$ bundles of $q$ coplanar, concurrent lines; the planes containing each bundle are the $q-1$ planes through $s$ not containing $M_1$ or $M_2$, and the points of concurrence of the bundles form, together with $M_1$ and $M_2$, an irreducible conic $\mathcal{C}$ of $\beta$ (Lemma \ref{Lm_IntersectionPointsAreConic}).
    As $\ell_{31}$ is skew to $s$ and is not contained in $\beta$ (nor contains $M_1$ or $M_2$), the line $\ell_{31}$ meets at most one line per bundle.
    By choice of $\ell_{31}$ (see Configuration \ref{Conf_FiveLines}), this line is skew to all lines of at least one bundle.
    In conclusion, there are at most $q-2$ planes in $\Pi_1$, one of which intersects $\Sigma_3$ in the line $\vspan{M_3,S}$.
    
    Now consider the planes of $\Pi_2$.
    By definition of $\ell_{11}''$ and $\ell_{21}''$, each line connecting a point of $\ell_{11}''\setminus\set{M_1}$ with a point of $\ell_{21}''\setminus\set{M_2}$ defines a unique plane of $\Pi^{(4)}$ that intersects $\pi_{12}$ in precisely the point $P_{12}$.
    Of these $q^2$ planes, only $q$ intersect $\ell_{31}$ (thus $|\Pi_2|=q$) and hence are part of a regulus of the unique hyperbolic quadric $\mathcal{Q}$ defined by the pairwise disjoint lines $\ell_{11}''$, $\ell_{21}''$ and $\ell_{31}$.
    
    \bigskip
    Let $e$ be an external line to $\mathcal{C}$ in $\beta$ through $M_3$ (note that this always exists, as $M_3$ lies on the $2$-secant $m$ to $\mathcal{C}$ and hence can never be a nucleus) and define the plane $\delta:=\vspan{e,Q}$.
    We claim that $\delta$ intersects $\mathcal{Q}$ in an irreducible conic.
    Note that as $M_1,M_2,M_3\in\mathcal{Q}$, the line $m$ is a generator of $\mathcal{Q}$ through $M_3$.
    The second generator of $\mathcal{Q}$ through $M_3$ is $\ell_{31}$.
    None of these two generators are contained in $\delta$, hence there does not exist a generator of $\mathcal{Q}$ that is contained in $\delta$ and goes through $M_3\in\mathcal{Q}$, implying that $\delta\cap\mathcal{Q}$ must be an irreducible conic.
    
    Observe that all planes of $\Pi_1$ intersect $\delta$ in at most a point.
    After all, if this would not be the case, an intersection line of a plane of $\Pi_1$ with $\Sigma_3$ would lie in $\delta$.
    Such intersection line also contains a point of the conic $\mathcal{C}$.
    However, the plane $\delta$ intersects the plane $\beta$ in the external line $e$ to $\mathcal{C}$, a contradiction.
    
    Note that, as said before, precisely one of the planes of $\Pi_1$ intersects $\Sigma_3$ in a line going through $M_3\in\delta$ and hence intersects $\delta$ in precisely that point.
    However, $M_3$ is already contained in $\mathcal{Q}$.
    In conclusion, all planes of $\Pi_1\cup\Pi_2$ intersect the plane $\delta$ in a point set $\mathcal{P}$ consisting of all $q+1$ points of an irreducible conic containing $M_3$ (originating from $\Pi_2$), together with at most $q-3$ extra points (originating from $\Pi_1$).
    By Theorem \ref{Res_BlockingExternalLines}, we can choose a line $\ell_{32}$ in $\delta$ that avoids all points of $\mathcal{P}\cup\set{Q}$.
    As $\ell_{32}\subseteq\Sigma_3$ is consequently skew to the line $m\ni M_3$ (as $m\nsubseteq\delta$), this line is skew to all planes of $\Pi_1\cup\Pi_2\cup\Pi_4$.
    
    \bigskip
    We claim that $\ell_{32}$ is skew to all planes of $\Pi_3$ as well, finishing the proof.
    Suppose that $\alpha\in\Pi_3$.
    Note that $\alpha\nsubseteq\Sigma_3$ as else it has to contain the points $M_1$, $M_2$, $P_{13}$ and $P_{23}$ to be able to intersect the lines $\ell_{11}$, $\ell_{12}$, $\ell_{21}$ and $\ell_{22}$, but those points are not coplanar.
    Hence, for each $i\in\set{1,2,3}$, $\alpha$ intersects $\Sigma_i$ in a line $a_i$.
    Suppose that $\alpha$ intersects $\pi_{12}$ in precisely the point $M_j$ for a $j\in\set{1,2}$ (which implies that $a_1\neq a_2$).
    Then, for every $i\in\set{1,2}$, the line $a_i$ intersects $\ell_{i2}$ in a point $Q_i$.
    Hence, the plane $\alpha$ contains two distinct points $Q_1$ and $Q_2$ of the plane $\vspan{\ell_{12},\ell_{22}}$ and hence has to intersect the line $s$, which means that the line $a_3$ has to intersect the line $s$.
    As $a_3$ has to intersect the line $\ell_{31}$ as well, it has to be contained in the plane $\vspan{M_j,\ell_{31}}$, which intersects the line $s$ in $Q$; thus $a_3$ has to go through $Q$.
    In conclusion, as $a_3$ is not contained in $\delta$ (because $M_j\notin\delta$), it has to intersect $\delta$ in precisely the point $Q$, which gets avoided by the line $\ell_{32}$.
\end{proof}

As a consequence, we immediately get the following.

\begin{crl}\label{Crl_SixHigPigPlanes4}
    There exist six planes of $\pg{4}{q}$ in higgledy-piggledy arrangement, two of which intersect in a line.
\end{crl}
\begin{proof}
    If $q\in\set{2,3,4,5}$, some computer-assisted searches prove the statement; see Code Snippet \ref{Snip_Planes4}.
    If $q\geq7$, the corollary follows immediately from Theorem \ref{Thm_SixHigPigLines4} and Proposition \ref{Prop_Duality}.
\end{proof}

\section{Higgledy-piggledy sets of \texorpdfstring{$\boldsymbol{\pgTitle{5}{q}}$}{PG(5,q)}}\label{Sect_HigPig5}

As an addition to Section \ref{Sect_HigPig4}, we briefly discuss small higgledy-piggledy sets of $k$-subspaces, $k\in\set{1,2,3}$.
There exists an optimal higgledy-piggledy line set in $\pg{5}{q}$ as the case $k=1$ was already considered in \cite{BartoliCossidenteMarinoPavese} (see Theorem \ref{Res_SporadicExamples}($3.$)).
As a consequence, the case $k=3$ is solved as well.

\begin{crl}\label{Crl_SevenHigPigSolids5}
    There exist seven solids of $\pg{5}{q}$, $q\geq7$, in higgledy-piggledy arrangement.
\end{crl}
\begin{proof}
    This follows immediately from Theorem \ref{Res_SporadicExamples}($3.$) and Proposition \ref{Prop_Duality}.
\end{proof}

This set is an optimal higgledy-piggledy set of solids in $\pg{5}{q}$.
Note that, by Proposition \ref{Prop_OptimalHigPigN-2}, these seven solids necessarily pairwise intersect in a line and that these intersection lines are pairwise disjoint.

\bigskip
The only remaining non-trivial case is $k=2$.
If we assume that $q\geq7$, then, by Theorem \ref{Thm_LowerBound}, a higgledy-piggledy set of planes in $\pg{5}{q}$ has size at least seven.
By Theorem \ref{Res_HigPigExistenceGeneral}, we know that there exists a higgledy-piggledy set of planes in $\pg{5}{q}$ of size ten.
All in all, there still seems to be room for improvement.
We will prove the existence of eight pairwise disjoint planes of $\pg{5}{q}$ in higgledy-piggledy arrangement (Theorem \ref{Thm_EightHigPigPlanes5}) by making use of field reduction (see Section \ref{Sect_ConstructionsAndExtremalSets}).

\begin{df}\label{Def_Subline}
    Let $m\in\NN\setminus\set{0}$.
    An \emph{$\FF_q$-subline} of $\pg{1}{q^m}$ is a point set isomorphic to $\pg{1}{q}$.
\end{df}

By considering the underlying vector space of $\pg{1}{q^m}$, one can easily see that each three distinct points define a unique $\FF_q$-subline of $\pg{1}{q^m}$.
Throughout this section, we mainly focus on the case $m=3$.

\bigskip
We remind the reader about the field reduction map $\mathcal{F}_{\n'+1,k+1,q}$ introduced in \eqref{Eq_FieldRed}, and consider this map for $\n'=1$ and $k=2$:
\[
    \mathcal{F}_{2,3,q}:\pg{1}{q^3}\rightarrow\pg{5}{q}\textnormal{,}
\]
with $\mathcal{D}_{2,3,q}$ the corresponding Desarguesian spread of planes in $\pg{5}{q}$.

\bigskip
Let $\mathcal{P}$ be an $\FF_q$-linear set of $\pg{1}{q^3}$.
Then precisely one of the following holds.
\begin{enumerate}
    \item $\mathcal{P}$ has rank $0$ and $|\mathcal{P}|=0$.
    \item $\mathcal{P}$ has rank $1$ and $|\mathcal{P}|=1$.
    \item $\mathcal{P}$ has rank $2$ and $\mathcal{P}$ is an $\FF_q$-subline.
    \item $\mathcal{P}$ has rank $3$ and $|\mathcal{P}|=q^2+1$; in this case, $\mathcal{P}$ will be called an \emph{$\FF_q$-club}.
    \item $\mathcal{P}$ has rank $3$ and $|\mathcal{P}|=q^2+q+1$; in this case, $\mathcal{P}$ will be called a \emph{scattered $\FF_q$-linear set}.
    \item $\mathcal{P}$ has rank $4$ and $|\mathcal{P}|=q^3+1$.
\end{enumerate}

These definitions correspond to the ones arising in the literature.
It is somewhat natural to consider these six types of $\FF_q$-linear sets of $\pg{1}{q^3}$, as the authors of \cite[Theorem $5$]{LavrauwVandeVoorde1} have proven that all sets of each type are projectively equivalent.
The authors of \cite{LavrauwVandeVoorde1} present several other results, which pave many paths towards our goal.
We select precisely those results that will lead us there the fastest.

The following result was already presented in \cite{FancsaliSziklai1} but, as the authors of \cite{LavrauwVandeVoorde1} point out, the proof was incomplete as they assumed the projective equivalence of $\FF_q$-clubs, which wasn't proven up to that point and is generally not true for $\FF_q$-clubs of $\pg{1}{q^m}$, $m\geq4$ \cite[Theorem $5$]{LavrauwVandeVoorde1}.

\begin{thm}[{\cite[Theorem $8$]{LavrauwVandeVoorde1}}]\label{Res_LinearSetSubline}
    An $\FF_q$-linear set intersects an $\FF_q$-subline of $\pg{1}{q^3}$ in $0$, $1$, $2$, $3$ or $q+1$ points.
\end{thm}

The next result has a story similar to the one above, as this result was first proven in \cite{FerretStorme} for $\FF_q$-linear sets of $\pg{1}{q^3}$, $q=p^h$, $p\geq7$, based on the projective equivalence of $\FF_q$-clubs and scattered $\FF_q$-linear sets.
Although the original theorem of \cite{LavrauwVandeVoorde1} concerns $\FF_q$-linear sets of $\pg{1}{q^m}$, we simplify their result to fit our needs.

\begin{thm}[{\cite[Theorem $23$]{LavrauwVandeVoorde1}}]\label{Res_LinearSetIntersection}
    Two distinct $\FF_q$-linear sets of rank at most $3$ in $\pg{1}{q^3}$, $q\geq4$, share at most $2q+3$ points.
\end{thm}

The following theorem presents the main result of this section.

\begin{thm}\label{Thm_EightHigPigPlanes5}
    There exist eight pairwise disjoint planes of $\pg{5}{q}$ in higgledy-piggledy arrangement.
\end{thm}
\begin{proof}
    First of all, the cases $q\in\set{2,3,4,5}$ can be checked by computer (see Code Snippet \ref{Snip_Planes5}); hence assume that $q\geq7$.
    
    Consider an $\FF_q$-subline $\mathfrak{b}_1$ of $\pg{1}{q^3}$ and let $C$, $B_{12}$, $B_{13}$ and $D_1$ be four distinct points lying on this subline.
    Take a point $B_{23}\notin\mathfrak{b}_1$ and define $\mathfrak{b}_2$ to be the unique $\FF_q$-subline containing the points of the set $\set{C,B_{12},B_{23}}$; let $D_2$ be a point of $\mathfrak{b}_2\setminus\set{C,B_{12},B_{23}}$.
    Naturally, $\mathfrak{b}_1\neq\mathfrak{b}_2$ as $B_{23}\notin\mathfrak{b}_1$.
    Finally, denote with $\mathfrak{b}_3$ the unique $\FF_q$-subline containing the points of the set $\set{C,B_{13},B_{23}}$.
    Note that $\mathfrak{b}_3\neq\mathfrak{b}_i$, $i\in\set{1,2}$, as else $\mathfrak{b}_i$ would contain $B_{(3-i)3}$, which would imply that $\mathfrak{b}_1$ and $\mathfrak{b}_2$ share three distinct points and hence would be equal, a contradiction.
    Take a point $D_3\in\mathfrak{b}_3\setminus\set{C,B_{13},B_{23}}$.
    
    In this way, we obtain three distinct $\FF_q$-sublines that pairwise intersect in two points and have the point $C$ in common.
    Define $\mathcal{P}:=\set{C,B_{12},B_{13},B_{23},D_1,D_2,D_3}$.
    By Theorem \ref{Res_LinearSetSubline}, any $\FF_q$-linear set that contains all points of $\mathcal{P}$ has to contain all points of $\mathfrak{b}_1\cup\mathfrak{b}_2\cup\mathfrak{b}_3$, as such an $\FF_q$-linear set contains at least four points of each subline.
    As $|\mathfrak{b}_1\cup\mathfrak{b}_2\cup\mathfrak{b}_3|=3q-2>2q+3$, Theorem \ref{Res_LinearSetIntersection} implies that there exists at most one $\FF_q$-linear set $\mathcal{P}_\textnormal{lin}$ of rank at most $3$ that contains all points of $\mathcal{P}$.
    Choose a point $Q\notin\mathcal{P}_\textnormal{lin}$.
    Then $\mathcal{P}\cup\set{Q}$ is a set of eight points in $\pg{1}{q^3}$ that is not contained in any $\FF_q$-linear set of rank at most $3$.
    Theorem \ref{Thm_FieldRed} finishes the proof.
\end{proof}

\subsubsection*{The optimal but unsolved case of seven planes in higgledy-piggledy arrangement}

Ideally, one would like to find seven planes of $\pg{5}{q}$ in higgledy-piggledy arrangement, as Theorem \ref{Thm_LowerBound} implies that, if $q\geq7$, no higgledy-piggledy set of six planes exists.
In fact, even more can be said of this theoretically smallest higgledy-piggledy plane set.

\begin{prop}
    Let $q\geq7$.
    Then any seven planes of $\pg{5}{q}$ in higgledy-piggledy arrangement are pairwise disjoint.
\end{prop}
\begin{proof}
    Let $\mathcal{K}:=\set{\pi_1,\pi_2,\dots,\pi_7}$ be the higgledy-piggledy set in question and suppose to the contrary (and w.l.o.g.) that there exists a hyperplane $\Sigma$ containing $\pi_1$ and $\pi_2$.
    Define $\ell_3$ and $\ell_4$ to be lines contained in $\pi_3\cap\Sigma$ and $\pi_4\cap\Sigma$, respectively, and let $\Pi$ be a solid in $\Sigma$ that contains $\vspan{\ell_3,\ell_4}$; choose a point $P_i$ in $\Pi\cap\pi_i$ for every $i\in\set{5,6,7}$.
    Then any plane $\pi\subseteq\Pi$ that contains $\vspan{P_5,P_6,P_7}$ naturally contains a point of $\pi_5$, $\pi_6$ and $\pi_7$.
    Moreover, as $\pi$ is contained in $\Pi\supseteq\ell_3,\ell_4$, this plane intersects $\pi_3$ and $\pi_4$ as well.
    Finally, as $\pi\subseteq\Sigma$, we conclude that $\pi$ meets all planes of $\mathcal{K}$, contradicting Theorem \ref{Res_CharHigPig}.
\end{proof}

Although the question whether a higgledy-piggledy plane set of size seven exists is still open, the observation above may hint us to try finding such a set as part of a Desarguesian spread, mimicking the strategy of proving Theorem \ref{Thm_EightHigPigPlanes5}.
One could indeed improve Theorem \ref{Thm_EightHigPigPlanes5} if there would exist three distinct $\FF_q$-sublines of $\pg{1}{q^3}$ that pairwise intersect in two points but have no point in common (implying the existence of a point set in $\pg{1}{q^3}$ of seven points not contained in any $\FF_q$-linear set of rank at most $3$).
The answer to this question of existence is, unfortunately, negative.

\begin{thm}\label{Thm_SpecialSublines}
    Let $m\in\NN\setminus\set{0}$ and $q\neq2$.
    Then there exist three distinct $\FF_q$-sublines $\mathfrak{b}_1$, $\mathfrak{b}_2$ and $\mathfrak{b}_3$ of $\pg{1}{q^m}$ with the property that
    \begin{enumerate}
        \item $|\mathfrak{b}_i\cap\mathfrak{b}_j|=2$ for every $i\neq j$, and
        \item $\mathfrak{b}_1\cap\mathfrak{b}_2\cap\mathfrak{b}_3=\emptyset$,
    \end{enumerate}
    if and only if $m$ is even.
\end{thm}
\begin{proof}
    First, suppose that three such $\FF_q$-sublines of $\pg{1}{q^m}$ do exist.
    Choose a coordinate system for the projective line and let $P_{01}$, $P_{10}$ and $P_{11}$ be the points corresponding to the coordinates $(0,1)$, $(1,0)$ and $(1,1)$, respectively.
    Without loss of generality, we may assume that $\set{P_{01},P_{10},P_{11}}\subseteq\mathfrak{b}_1$, $\set{P_{01},P_{10}}\subseteq\mathfrak{b}_2$ and $P_{11}\in\mathfrak{b}_3$, the first assumption implying that all points of $\mathfrak{b}_1$ correspond to the set of coordinates $\set{(0,1)}\cup\sett{(1,a)}{a\in\FF_q}$.
    
    By considering an element of $\textnormal{PGL}(2,q)$ that maps three distinct point of $\mathfrak{b}_1$ (e.g.\ take $P_{01}$, $P_{10}$ and $P_{11}$) onto three distinct points of $\mathfrak{b}_i$, $i\in\set{2,3}$, one can find the set of coordinates corresponding to the points on the subline $\mathfrak{b}_i$.
    In this way, we know that there exists an $\alpha\in\FF_{q^m}\setminus\FF_q$ such that the points of $\mathfrak{b}_2$ correspond to the set of coordinates
    \begin{equation}\label{Eq_SecondSubline}
        \set{(0,1)}\cup\sett{(1,b\alpha)}{b\in\FF_q}\textnormal{.}
    \end{equation}
    Suppose that the unique point $P_{1a_0}\in(\mathfrak{b}_1\cap\mathfrak{b}_3)\setminus\set{P_{11}}$ has coordinates $(1,a_0)$, $a_0\in\FF_q\setminus\set{0,1}$.
    Then, analogously, there exists a $\beta\in\FF_{q^m}\setminus\FF_q$ such that the points of $\mathfrak{b}_3$ correspond to the set of coordinates
    \begin{equation}\label{Eq_ThirdSubline}
        \set{(1,1)}\cup\sett{(c+\beta,c+a_0\beta)}{c\in\FF_q}\textnormal{.}
    \end{equation}
    By the given properties of these three sublines, there should exist two points, \emph{not equal to $P_{10}$ or $P_{1a_0}$}, with coordinates contained in both sets \eqref{Eq_SecondSubline} and \eqref{Eq_ThirdSubline}.
    As $c+\beta\neq0$ for any $c\in\FF_q$, this is equivalent with stating that the equality
    \begin{equation}\label{Eq_Equality}
        b\alpha=\frac{c+a_0\beta}{c+\beta}
    \end{equation}
    has two distinct solutions $(b_1,c_1)$ and $(b_2,c_2)$, with $b_1,b_2,c_1,c_2\in\FF_q\setminus\set{0}$.
    Plugging each solution into \eqref{Eq_Equality} and solving for $\beta$, we obtain that $c_1\frac{b_1\alpha-1}{a_0-b_1\alpha}=\beta=c_2\frac{b_2\alpha-1}{a_0-b_2\alpha}$, which expands to
    \begin{equation}\label{Eq_AlphaIsDependant}
        b_1b_2(c_1-c_2)\alpha^2=\big(a_0(b_1c_1-b_2c_2)+b_2c_1-b_1c_2\big)\alpha+a_0(c_2-c_1)\textnormal{.}
    \end{equation}
    The elements $b_1$ and $b_2$ are non-zero.
    If $(c_1-c_2)$ would be zero, then, as $c_1,c_2\neq0$, expression \eqref{Eq_AlphaIsDependant} reduces to
    \[
        (a_0-1)(b_1-b_2)=0\textnormal{,}
    \]
    which cannot be valid as $a_0\in\FF_q\setminus\set{0,1}$ and $(b_1,c_1)\neq(b_2,c_2)$, leading to a contradiction.
    In conclusion, expression \eqref{Eq_AlphaIsDependant} implies that $\FF_q[\alpha]$ is the subfield of $\FF_{q^m}$ isomorphic to $\FF_{q^2}$, which can only be true if $m$ is even.
    
    Conversely, assume $m$ even.
    Consider a point set of $\pg{1}{q^m}$ isomorphic to $\pg{1}{q^2}$.
    Then this point set, together with all $\FF_q$-sublines of $\pg{1}{q^2}$, can be identified as the point set of an irreducible elliptic quadric $\mathcal{Q}^-(3,q)$ of $\pg{3}{q}$, together with its non-tangent plane intersections\footnote{This is a particular ovoidal circle geometry called a \emph{M\"obius plane}.} \cite[Lemma $17.1.5$]{Hirschfeld}.
    
    Take a line $s_0$ intersecting $\mathcal{Q}^-(3,q)$ in two points $Q_1$ and $Q_2$, and consider the $q+1$ planes through $s_0$.
    Each such plane $\pi$ intersects $\mathcal{Q}^-(3,q)$ in an irreducible conic $\mathcal{C}_\pi$, and there always exists a point $S_\pi\in s_0\setminus\set{Q_1,Q_2}$ that lies on at most one tangent line to that conic.
    By the pigeon hole principle, we can choose two planes $\pi_1$ and $\pi_2$ through $s_0$ such that there exists a point $S\in s_0\setminus\set{Q_1,Q_2}$ that lies on at most one tangent line to both $\mathcal{C}_{\pi_1}$ and $\mathcal{C}_{\pi_2}$ (if $q=3$ then $S$ necessarily lies on no tangent lines to both $\mathcal{C}_{\pi_1}$ and $\mathcal{C}_{\pi_2}$).
    As $q\geq3$, we can now take a line $s_i\neq s_0$ through $S$ that intersects $\mathcal{C}_{\pi_i}$ in two points, $i\in\set{1,2}$; define $\pi_3:=\vspan{s_1,s_2}$.
    One can easily check that the $\FF_q$-sublines of $\pg{1}{q^2}$ corresponding to the intersections of $\pi_1$, $\pi_2$ and $\pi_3$ with $\mathcal{Q}^-(3,q)$ meet the requirements.
\end{proof}

The above theorem does not, however, eliminate the chance of finding seven planes of $\pg{5}{q}$ in higgledy-piggledy arrangement using the field reduction method.
Computer-assisted searches confirm that such small subsets of a Desarguesian $2$-spread exist for $q\in\set{2,3,4,5,7}$ (Code Snippet \ref{Snip_Planes5}).
Hence, we carefully suspect that there generally exist seven planes in higgledy-piggledy arrangement as part of a Desarguesian spread of $\pg{5}{q}$.

\begin{prob}
    Prove that there exist seven planes of $\pg{5}{q}$ in higgledy-piggledy arrangement.
\end{prob}

Once we get a better grasp on the structure of all $\FF_q$-linear sets in $\pg{1}{q^3}$, the above open problem might be solvable using Theorem \ref{Thm_FieldRed}.

\section{Minimal codes, covering codes and resolving sets}\label{Sect_Fruits}

This section is aimed to briefly discuss the applications of the results of Section \ref{Sect_HigPig4} and \ref{Sect_HigPig5} to coding and graph theory.
Let $\len\in\NN$ and $k,\red\in\set{0,1,\dots,\len}$.
Any subspace of $\FF_q^\len$ of dimension $k=\len-\red$ or codimension (redundancy) $\red$ is said to be a \emph{($q$-ary) linear code} of length $\len$ and dimension $k$, and will be called a \emph{linear $[\len,k]_q$-code}.
Elements of such a code are called \emph{codewords}; the \emph{support} of a codeword is the set of indices in which the codeword has non-zero entries.
A code is called \emph{non-degenerate} if each index $i$ is contained in at least one support of a codeword.
See e.g.\ \cite{AssmusKey} for an introduction into the topic of coding theory.

\subsubsection*{Short minimal linear codes of dimension $\boldsymbol{5}$}

The new results of Section \ref{Sect_HigPig4} can be directly translated to upper bounds on the length of certain minimal codes.

\begin{df}
    A codeword of a linear code is called \emph{minimal} if its support contains no support of any other codeword except for its scalar multiples.
    A linear code is \emph{minimal} if all its codewords are minimal.
\end{df}

The authors of \cite{AlfaranoBorelloNeri,TangQiuLiaoZhou} independently proved that minimal codes have a one-to-one correspondence to strong blocking sets (alternatively called cutting blocking sets).
Recently, this correspondence was reproven geometrically in \cite[Corollary $3.3$]{HegerNagy}.

\begin{thm}[{\cite[Theorem $3.4$, Theorem $14$]{AlfaranoBorelloNeri,TangQiuLiaoZhou}}]\label{Res_CharMinimalCodes}
    Let $\mathcal{C}$ be a non-degenerate linear $[\len,k]_q$ code with generator matrix $G=(G_1,\dots,G_\len)$.
    Let $\Strong=\set{G_1,\dots,G_\len}$ be the corresponding point set of $\pg{k-1}{q}$.
    Then $\mathcal{C}$ is a minimal code if and only if $\Strong$ is a strong blocking set.
\end{thm}

As one is generally interested in the smallest possible length of minimal $q$-ary linear $[\len,k]_q$ codes for fixed parameters $k$ and $q$, one defines $m(k,q)$ to be the smallest possible length of such a code.
The following theorem bundles all relevant known results concerning the case of $k=5$.

\begin{thm}[\cite{AlfaranoBorelloNeriRavagnani,BartoliKissMarcuginiPambianco,FancsaliSziklai1}]\label{Res_FruitMinimal}
    \[
        4q+4\leq m(5,q)\leq\begin{cases}6q+6\quad&\textnormal{if }q>36086\textnormal{ and }2,3\nmid q\textnormal{,}\\
        7q+7\quad&\textnormal{if }q\geq7\textnormal{,}\\
        8q-3\textnormal{.}&\end{cases}
    \]
\end{thm}
\begin{proof}
    The lower bound on $m(5,q)$ is proven in \cite[Theorem $2.14$]{AlfaranoBorelloNeriRavagnani} for general $k$, and reproven geometrically in \cite[Theorem $3.9$]{HegerNagy}.
    This lower bound can in fact be improved by $1$ if $q\geq9$ (\cite[Corollary $2.19$]{AlfaranoBorelloNeriRavagnani}).
    
    The first two upper bounds arise by combining respectively Theorem \ref{Res_SporadicExamples}($2.$) and Theorem \ref{Res_HigPigExistence} with Theorem \ref{Res_CharMinimalCodes}.
    The third upper bound is proven in \cite[Construction $2$]{AlfaranoBorelloNeriRavagnani}; in this work, the authors prove the existence of eight lines of $\pg{4}{q}$ in higgledy-piggledy arrangement.
\end{proof}

Furthermore, the authors of \cite{AlfaranoBorelloNeriRavagnani} computationally proved that $m(5,2)=13$ and $m(5,3)\leq20$.
Our result concerning this topic comes down to the following.

\begin{thm}\label{Thm_FruitMinimal}
    $m(5,q)\leq6q+5$.
\end{thm}
\begin{proof}
    Directly from Theorem \ref{Thm_SixHigPigLines4} and \ref{Res_CharMinimalCodes}.
\end{proof}

It's easy to check that Theorem \ref{Thm_FruitMinimal} improves the existing upper bounds on $m(5,q)$ for all $q\geq5$.

\subsubsection*{Short covering codes of codimension $\boldsymbol{5}$ with covering radius $\boldsymbol{4}$}

Existence results on (small) strong $k$-blocking sets lead to existence results on (small) \emph{$(\n-k)$-saturating sets}, which in turn imply existence results on (short) \emph{covering codes}.
Allow us to introduce these notions.

\begin{df}
	Let $\Sat$ be a point set of $\pg{\n}{q}$.
	\begin{enumerate}
		\item A point $P\in\pg{\n}{q}$ is said to be $\rho$-\emph{saturated} by $\Sat$ (or, conversely, the set $\Sat$ $\rho$-\emph{saturates} $P$) if there exists a subspace through $P$ of dimension at most $\rho$ that is spanned by points of $\Sat$.
		\item The set $\Sat$ is a $\rho$-\emph{saturating set} of $\pg{\n}{q}$ if $\rho$ is the least integer such that all points of $\pg{\n}{q}$ are $\rho$-saturated by $\Sat$.
		Let $\satbound(\n,\rho)$ denote the smallest possible size of a $\rho$-saturating set of $\pg{\n}{q}$.
	\end{enumerate}
\end{df}

The authors of \cite[Theorem $3.2$]{DavydovEtAl} proved that strong $k$-blocking sets of an embedded $\pg{\n}{q}$ are $(\n-k)$-saturating sets of the ambient geometry $\pg{\n}{q^{\n-k+1}}$.
The author of \cite{Denaux} described this method of constructing $(\n-k)$-saturating sets as the \emph{strong blocking set approach}.

There has been done a lot of research concerning $\rho$-saturating sets in $\pg{\n}{q}$.
The following bundles relevant known results concerning $\satboundpow{4}(4,3)$, $\satboundpow{3}(4,2)$, $\satboundpow{4}(5,3)$ and $\satboundpow{3}(5,2)$.

\begin{thm}[\cite{DavydovEtAl,DavydovOstergard,Denaux}]
    In the following, $e\approx2.718...$ depicts Euler's number.
    \begin{enumerate}
        \item $\frac{4}{e}q+\frac{3}{2}<\satboundpow{4}(4,3)\leq\begin{cases}6q+6\quad&\textnormal{if }q>36086\textnormal{ and }2,3\nmid q\textnormal{,}\\
        7q+7\quad&\textnormal{if }q\geq7\textnormal{,}\\
        8q-3\textnormal{.}&\end{cases}$
        \item $\frac{3}{e}q^2+1<\satboundpow{3}(4,2)\leq6q^2+3q-6$.
        \item $\frac{4}{e}q^2+\frac{3}{2}<\satboundpow{4}(5,3)\leq4q^2+4q+4$.
        \item $\frac{3}{e}q^3+1<\satboundpow{3}(5,2)\leq3q^3+1$.
    \end{enumerate}
\end{thm}
\begin{proof}
    The lower bound arises from \cite[Proposition $4.2.1$]{Denaux}.
    The upper bound on $\satboundpow{4}(4,3)$ is the same as the one on $m(5,q)$ (see Theorem \ref{Res_FruitMinimal}).
    The upper bound on $\satboundpow{3}(4,2)$ follows from \cite[Theorem $7.2.9$]{Denaux} if $q\neq2$ and \cite[Theorem $3.16$]{DavydovEtAl} if $q=2$.
    The upper bound on $\satboundpow{4}(5,3)$ arises from \cite[Corollary $7.2$]{DavydovEtAl}, and, finally, the upper bound on $\satboundpow{4}(5,2)$ was obtained in \cite[Theorem $7$]{DavydovOstergard}.
\end{proof}

The computations of \cite{AlfaranoBorelloNeriRavagnani} furthermore imply that $s_{16}(4,3)\leq13$ and $s_{81}(4,3)\leq20$.

\bigskip
If we translate our results (Theorem \ref{Thm_SixHigPigLines4}, Corollary \ref{Crl_SixHigPigPlanes4}, Theorem \ref{Thm_EightHigPigPlanes5} and Corollary \ref{Crl_SevenHigPigSolids5}) to the context of saturating sets using the strong blocking set approach, we respectively obtain the following.
    \begin{enumerate}
        \item $\satboundpow{4}(4,3)\leq6q+5$.
        \item $\satboundpow{3}(4,2)\leq\begin{cases}6q^2+5q+1&\textnormal{if }q\leq5\textnormal{,}\\6q^2+5q-9&\textnormal{if }q\geq7\textnormal{.}\end{cases}$
        \item $\satboundpow{4}(5,3)\leq8q^2+8q+8$.
        \item $\satboundpow{3}(5,2)\leq7q^3+7q^2-14q-14\quad$if $q\geq7$.
    \end{enumerate}
Note that for bound $2.$ and $4.$ we made use of Proposition \ref{Prop_OptimalHigPigN-2} for $q\geq7$.
If $q\leq5$, we considered the worst-case scenarios of point coverage for six planes of $\pg{4}{q}$, two of which intersect in a line (the other four planes each contribute at most $q^2+q$ points).

Unfortunately, only one of these results improves the ones from the literature (if $q\geq5$), which we repeat below.
\begin{thm}\label{Thm_FruitSaturating}
    $\satboundpow{4}(4,3)\leq6q+5$.
\end{thm}

We can now translate this result to the coding-theoretical context.

\bigskip
The \emph{Hamming distance} between two vectors of $\FF_q^\len$ equals the number of positions in which they differ.
A $q$-ary linear code of length $\len$ and codimension (redundancy) $\red$ is said to have \emph{covering radius} $R$ if $R$ is the least integer such that every vector of $\FF_q^\len$ lies within Hamming distance $R$ of a codeword.
Whenever linear codes are investigated with the goal of optimising the length or (co)dimension with respect to the covering radius, such codes are often called $[n,n-r]_qR$ \emph{covering} codes.
This type of $q$-ary linear codes have a wide range of applications; for a description of several examples of such applications, see \cite[Section $1$]{DavydovEtAl}.

Suppose that $\Sat$ is a point set of $\pg{\red-1}{q}$ of size $\len$ and let $H$ be a $q$-ary $(\red\times\len)$-matrix with the homogeneous coordinates of the points of $\Sat$ as columns.
Then $\Sat$ is an $(R-1)$-saturating set of $\pg{\red-1}{q}$ if and only if $H$ is a parity check matrix of an $[\len,\len-\red]_qR$ code.
This describes a one-to-one correspondence between saturating sets of projective spaces and linear covering codes.
More specifically, any $\rho$-saturating set $\Sat$ of $\pg{\n}{q}$ corresponds to an $[\len,\len-\red]_qR$ code with
\[
    \len=|\Sat|\textnormal{,}\qquad\red=\n+1\quad\textnormal{and}\quad R=\rho+1\textnormal{.}
\]

Due to this correspondence, finding small $\rho$-saturating sets in $\pg{\n}{q}$ is equivalent to finding $[\len,\len-\red]_qR$ codes of small length.
In light of this, the \emph{length function} $\lenfunc(\red,R)$ is the smallest length of a $q$-ary linear code with covering radius $R$ and codimension $\red$.
Note that
\[
    \lenfunc(\red,R)=\satbound(\red-1,R-1)\textnormal{.}
\]

\begin{thm}
    $\lenfuncpow{4}(5,4)\leq6q+5$.
\end{thm}

\begin{rmk}
    The authors of \cite{DavydovEtAl} describe a strong tool called `$q^m$-concatenating constructions' to construct infinite families of covering codes with fixed covering radius $R$.
    One could consider to use this tool on the construction behind Theorem \ref{Thm_SixHigPigLines4}, Theorem \ref{Res_SporadicExamples} or even Theorem \ref{Res_HigPigExistence} to obtain families of short covering codes and study their \emph{(asymptotic) covering densities} (see \cite{DavydovEtAl}).
\end{rmk}

\subsubsection*{Small resolving sets of the point-hyperplane incidence graph}

Finally, some results can be deduced on the size of smallest resolving sets of the point-hyperplane incidence graph of $\pg{\n}{q}$.

\begin{df}
    Consider a finite, connected simple graph $\Gamma=(V,E)$ and let $d:V\times V\rightarrow\NN$ be its metric.
    A vertex $v\in V$ is called \emph{resolved} by a vertex set $S=\set{v_1,v_2,\dots,v_n}$ if the ordered sequence $(d(v,v_1),d(v,v_2),\dots,d(v,v_n))$ is unique.
    The set $S$ is called a \emph{resolving set} of $\Gamma$ if it resolves all its vertices.
\end{df}

Let $\Gamma_{\mathcal{P},\mathcal{H}}(\n,q)$ be the point-hyperplane incidence graph of $\pg{\n}{q}$, i.e.\ the bipartite graph that associates every point and every hyperplane with a vertex, vertices of different parts are adjacent if and only if the corresponding point lies in the corresponding hyperplane.
The authors of {\cite[Theorem $4$]{BartoliKissMarcuginiPambianco}} proved that if $q$ is large enough, any resolving set of $\Gamma_{\mathcal{P},\mathcal{H}}(\n,q)$ has size at least
\[
    2\n q-2\frac{\n^{\n-1}}{(\n-1)!}\textnormal{.}
\]

Another result obtained in \cite{BartoliKissMarcuginiPambianco} can be somewhat generalised, using the same arguments, as follows.

\begin{thm}[{\cite[Lemma $8$]{BartoliKissMarcuginiPambianco}}]\label{Res_ResolvingSet}
    Suppose that $\mathcal{L}=\set{\ell_1,\ell_2,\dots,\ell_k}$ is a higgledy-piggledy line set of $\pg{\n}{q}$ and $P_i\in\ell_i$ is a point not lying in any $\ell_j$, $j\neq i$; define
    \[
        m:=|\sett{P\in\ell_i\setminus\set{P_i}}{i\in\set{1,2,\dots,k}}|\textnormal{.}
    \]
    Then $\Gamma_{\mathcal{P},\mathcal{H}}(\n,q)$ has a resolving set of size $2m$.
\end{thm}

Hence, using the known results concerning existence of small higgledy-piggledy line sets, the same authors pointed out that $\Gamma_{\mathcal{P},\mathcal{H}}(3,q)$ has a resolving set of size $8q$ \cite[Theorem $10$]{BartoliKissMarcuginiPambianco}.
As a corollary of their main result (Theorem \ref{Res_SporadicExamples}($2.$)), they also proved that $\Gamma_{\mathcal{P},\mathcal{H}}(4,q)$ has a resolving set of size $12q$ if $q>36086$ is no power of $2$ or $3$ \cite[Corollary $13$]{BartoliKissMarcuginiPambianco}.
We can slightly extend and improve this result, as well as translate the existing result concerning higgledy-piggledy line sets of $\pg{5}{q}$ to this graph-theoretical context.

\begin{thm}
    The graph $\Gamma_{\mathcal{P},\mathcal{H}}(4,q)$ has a resolving set of size $12q-2$.
    The graph $\Gamma_{\mathcal{P},\mathcal{H}}(5,q)$ has a resolving set of size $14q$.
\end{thm}
\begin{proof}
    Directly from Theorem \ref{Res_SporadicExamples}($3.$) and \ref{Thm_SixHigPigLines4}, combined with Theorem \ref{Res_ResolvingSet}.
\end{proof}

\textbf{Acknowledgements.}
The author would like to credit Maarten De Boeck and the author's supervisor Leo Storme for providing a sketch of the proof of the respective necessary and sufficient conditions of Theorem \ref{Thm_SpecialSublines}.
Furthermore, a lot of gratitude goes towards the many proofreaders of this work.

\bibliographystyle{plain}
\bibliography{main.bib}

\appendix

\section{Relevant GAP-code for small values of \texorpdfstring{$\boldsymbol{q}$}{q}}

In this appendix, we list some relevant snippets of code that prove the existence of certain small higgledy-piggledy sets in $\pg{4}{q}$ and $\pg{5}{q}$, $q$ small.
The main tactic to tackle this problem is by randomly choosing subspaces and checking whether the selected set meets the desired property.
We use the package `FinInG' \cite{fining} of the GAP system \cite{GAP4}, hence one needs to call \texttt{LoadPackage("FinInG");} before executing any of the code below.

The following snippet checks whether for $q\in\set{2,3,4,5}$ there exist six higgledy-piggledy planes in $\pg{4}{q}$, two of which intersect in a line.

\begin{snip}\label{Snip_Planes4}\hfill

    \texttt{gap> for q in [2..5] do}
    
    \texttt{>\qquad\qquad pg := PG(4,q);}
    
    \texttt{>\qquad\qquad repeat}
    
    \texttt{>\qquad\qquad\qquad planeSet := [];}
    
    \texttt{>\qquad\qquad\qquad line := Random(Lines(pg));}
    
    \texttt{>\qquad\qquad\qquad for i in [1,2] do}
    
    \texttt{>\qquad\qquad\qquad\qquad AddSet(planeSet,Random(Planes(line)));}
    
    \texttt{>\qquad\qquad\qquad od;}
    
    \texttt{>\qquad\qquad\qquad for i in [1..4] do}
    
    \texttt{>\qquad\qquad\qquad\qquad AddSet(planeSet,RandomSubspace(pg,2));}
    
    \texttt{>\qquad\qquad\qquad od;}
    
    \texttt{>\qquad\qquad until ForAll(Planes(pg),pl -> Dimension(Span(List(planeSet,pl2}
    
    \texttt{\qquad\qquad\;\;-> Meet(pl,pl2)))) = 2);}
    
    \texttt{>\qquad\qquad Print("Test for q = ",q,": succes!\textbackslash n");}
    
    \texttt{>\qquad od;}
\end{snip}

The next the snippet checks whether there exist seven planes of $\pg{5}{q}$ in higgledy-piggledy arrangement as part of a Desarguesian spread, $q\in\set{2,3,4,5,7}$.

\begin{snip}\label{Snip_Planes5}\hfill

    \texttt{gap> for q in [2,3,4,5,7] do}
    
    \texttt{>\qquad\qquad pgLine := PG(1,q\^{}3);}
    
    \texttt{>\qquad\qquad pgBig := PG(5,q);}
    
    \texttt{>\qquad\qquad proj := NaturalEmbeddingByFieldReduction(pgLine,pgBig);}
    
    \texttt{>\qquad\qquad repeat}
    
    \texttt{>\qquad\qquad\qquad	pointSet := [];}
    
    \texttt{>\qquad\qquad\qquad	while Size(pointSet) < 7 do}
    
    \texttt{>\qquad\qquad\qquad\qquad AddSet(pointSet,Random(Points(pgLine)));}
    
    \texttt{>\qquad\qquad\qquad	od;}
    
    \texttt{>\qquad\qquad\qquad	planeSet := Set(pointSet,p -> p\^{}proj);}
    
    \texttt{>\qquad\qquad until ForAll(Solids(pgBig),sol -> Dimension(Span(List(planeSet,}
    
    \texttt{\qquad\qquad\;\;pl -> Meet(sol,pl)))) = 3);}
    
    \texttt{>\qquad\qquad Print("Test for q = ",q,": succes!\textbackslash n");}
    
    \texttt{>\qquad od;}
\end{snip}

\bigskip
Author's address:

\bigskip
Lins Denaux

Ghent University

Department of Mathematics: Analysis, Logic and Discrete Mathematics

Krijgslaan $281$ -- Building S$8$

$9000$ Ghent

BELGIUM

\texttt{e-mail : lins.denaux@ugent.be}

\texttt{website: }\url{https://users.ugent.be/~ldnaux}

\end{document}